\documentclass[12pt,leqno]{amsart}
\usepackage{latexsym,amsmath,amssymb,xcolor,mathrsfs,marvosym}
\usepackage{graphicx,hyperref}
\usepackage{esint}

\title{The Coarea Inequality}
\author{Behnam Esmayli, Piotr Haj{\l}asz}

\address{B.\ Esmayli: Department of Mathematics, University of Pittsburgh, Pittsburgh, PA 15260, USA, {\tt bee23@pitt.edu}}

\address{P.\ Haj{\l}asz: Department of Mathematics, University of Pittsburgh, 301
  Thackeray Hall, Pittsburgh, PA 15260, USA, {\tt hajlasz@pitt.edu}}
 \thanks{P.H. was supported by NSF grant  DMS-1800457.}

plus 6pt minus 12pt
\newtheorem{theorem}{Theorem}
\newtheorem{lemma}[theorem]{Lemma}
\newtheorem{corollary}[theorem]{Corollary}

\theoremstyle{definition}
\newtheorem{remark}[theorem]{Remark}
\newtheorem{definition}[theorem]{Definition}

\newcommand{\barint}{
\rule[.036in]{.12in}{.009in}\kern-.16in \displaystyle\int }

\newcommand{\barcal}{\mbox{$ \rule[.036in]{.11in}{.007in}\kern-.128in\int $}}

\newcommand{\bbbn}{\mathbb N}
\newcommand{\bbbz}{\mathbb Z}
\newcommand{\bbbr}{\mathbb R}
\newcommand{\eps}{\varepsilon}

\def\5{\text{\Saturn}}

\def\diam{\operatorname{diam}}

\def\H{{\mathcal H}}

\def\lip{{\rm Lip\,}}

\def\mvint_#1{\mathchoice
          {\mathop{\vrule width 6pt height 3 pt depth -2.5pt
                  \kern -8pt \intop}\nolimits_{\kern -3pt #1}}%
          {\mathop{\vrule width 5pt height 3 pt depth -2.6pt
                  \kern -6pt \intop}\nolimits_{#1}}%
          {\mathop{\vrule width 5pt height 3 pt depth -2.6pt
                  \kern -6pt \intop}\nolimits_{#1}}%
          {\mathop{\vrule width 5pt height 3 pt depth -2.6pt
                  \kern -6pt \intop}\nolimits_{#1}}}

\numberwithin{theorem}{section} \numberwithin{equation}{section}

\begin{document}

\sloppy


\begin{abstract}
The aim of this paper is to provide a self-contained proof of a general case of the coarea inequality, also known as the Eilenberg inequality. The result is known, but we are not aware of any place that a proof would be written with all details. The known proof is based on a difficult result of Davies. Our proof is elementary and does not use Davies' theorem. Instead we use an elegant argument that we learned from Nazarov through MathOverflow. We also obtain some generalizations of the coarea inequality. 
\end{abstract}

\maketitle

\section{Introduction}

The aim of this paper is to provide an elementary and self-contained proof of the following result which is known under the name of the {\em coarea inequality} or the  {\em Eilenberg inequality}. 
\begin{theorem}
\label{T14}
Let $X$ and $Y$ be arbitrary metric spaces, $0\leq t\leq s<\infty$ (any) real numbers and $E\subset X$ any subset. 
Then, for any Lipschitz map $ f: X \to Y$ we have
\begin{equation}
\label{eq29}
\int^*_Y \mathcal{H}^{s-t} \left(f^{-1}(y) \cap E\right)\,  d\mathcal{H}^t (y) \leq \left(\lip{f}\right)^t \, \frac{\omega_{s-t} \omega_t}{\omega_{s}}\mathcal{H}^{s}(E) \, .
\end{equation}
Moreover if $X$ is boundedly compact i.e., bounded and closed sets in $X$ are compact, $E$ is $\H^s$-measurable, 
and $\H^s(E)<\infty$, then the function
\begin{equation}
\label{eq28}
y\mapsto \mathcal{H}^{s-t} \left(f^{-1}(y) \cap E\right)
\end{equation}
is $\H^t$-measurable and therefore, the upper integral  can be replaced with the usual integral.
\end{theorem}
Here $\H^\alpha$ stands for the $\alpha$-dimensional Hausdorff measure and
$\int^* g \, d\mu$ is the upper integral which does not require measurability of the integrand.

\begin{remark}
In general, we cannot expect measurability of the function \eqref{eq28} as the following simple example shows:
Let $V\subset\bbbr$ be a non-measurable set. Let $X=V$, $Y=\bbbr$ and $f:X\to Y$, $f(x)=x$. Then for $s=t=1$, and $E=X$, the function \eqref{eq28} is the characteristic function of $V$ and therefore is not measurable.
It was communicated to us by Pertti Mattila \cite{mattila2} that
\eqref{eq28} is measurable with respect to the sigma-algebra generated by analytic sets if $X$ and $Y$ are Polish spaces and $E$ is analytic. This is a consequence of the work of Dellacherie \cite{del}, see Remark~7.8 in \cite{mattila}. 
However, we did not verify this statement.
\end{remark}

Proving measurability of \eqref{eq28} under the given assumptions is not difficult, see Section~\ref{NC}, and the main difficulty rests in proving inequality \eqref{eq29}. Thus in the discussion below we will focus on \eqref{eq29} only.

The inequality was first proved by Eilenberg \cite{eilenberg} in 1938 in the case when $t=1$, $Y=\bbbr$ and $f(\cdot)=d(\cdot,x_o)\to\bbbr$ is the distance to a point on a metric space $X$. Then it was generalized in \cite{eh} to the case of  $t=1$, $Y=\bbbr$ and $f:X\to\bbbr$ any Lipschitz function. 

It seems however, that a related argument was used by Szpilrajn\footnote{He changed his name to Marczewski while hiding from Nazi persecution.} \cite{szpilrajn} in the proof that if $\H^{n+1}(X)=0$, then  the topological dimension of $X$ is at most $n$. Szpilrajn's proof is reproduced in \cite[Theorem~7.3]{HW} and \cite[Theorem~8.15]{heinonen}. 
Szpilrajn mentions that his argument is based on N\"obeling's proof of a weaker result that the topological dimension is bounded from above by the Hausdorff dimension of a metric space \cite{nobeling} (N\"obeling's paper is reproduced in \cite{menger}). The reader may find a translation of N\"obeling's paper  in MathOverflow~\cite{MO2}, and it is clear that his argument was closely related to Eilenberg's inequality for the distance function.
From reading Szpilrajn's paper, it is also clear that there was a strong collaboration between him and Eilenberg. 

\begin{remark}
\label{R6}
Most of the proofs that the reader may find in the literature \cite[Theorem~13.3.1]{buragoz}, 
\cite[Lemma~5.2.4]{KP}, \cite[Theorem~7.7]{mattila}, apply to the case of Lipschitz mappings $f:X\to\bbbr^m$ and $t=m$, and the proofs do not differ much from that in \cite{eh}.
Since the proofs use the fact that for a subset $A\subset Y=\bbbr^m$, the isodiamteric inequality holds, that is $\H^m(A)\leq \omega_m(\diam A)^m/2^m$, there is no obvious way how such proofs could be generalized to other metric spaces $Y$.
\end{remark}
\begin{remark}
Regarding coarea inequality for mappings into metric spaces one should mention an interesting paper by Mal\'y \cite{maly}.
The result given in \cite[Proposition~3.1.5]{ambrosiot} covers the general case but, as confirmed by the authors,  the proof is incorrect.
\end{remark}

Proving the result in a more general case was a remarkable achievement of Federer \cite{federer2}, see also \cite[Theorem~2.10.25]{federer}. However, he could prove Theorem~\ref{T14} only under additional assumptions that 
\begin{itemize}
\item[(a)] The integrand $\H^{s-t}(f^{-1}(y)\cap E)$ is positive (only) on a set of $\sigma$-finite measure $\H^t$; or
\item[(b)] The space $Y$ is {\em boundedly compact}, meaning that bounded and closed sets are compact.
\end{itemize}
His strategy was as follows. 
He first proved an inequality 
more or less equivalent to 
(see Lemma~\ref{T17} and Remark~\ref{R5} below),
\begin{equation}
\label{eq24}
\int^\bullet_Y \H^{s-t}_{\delta} \left(f^{-1}(y) \cap E\right) \, d\H^t 
\leq \left(\lip{f}\right)^t \, \frac{\omega_{s-t} \omega_t}{\omega_{s}}\mathcal{H}^{s}(E)\, ,
\end{equation}
where the left-hand side is the \textit{weighted integral} (see Definition~\ref{D3}). Federer \cite[2.10.24]{federer} used however, different notation (see Remark~\ref{R4}).

This inequality follows from a straightforward covering argument. 
In fact the proof is very similar to the classical proof due to Eilenberg, the one the reader can find in \cite{buragoz,KP,mattila}, see Remark~\ref{R6}.

The coarea inequality then follows from the following theorem -- which is of independent interest  -- and a simple monotone convergence theorem for upper integrals as $\delta\to 0^+$.
\begin{theorem}
\label{T15}
Let $Y$ be an arbitrary metric space.
For $t \in [0,\infty)$, and any $g:Y \to [0,\infty]$ we have
$$
\int^*_Y g(y) \ d\mathcal{H}^t(y)=
\int^\bullet_Y g(y) \ d\mathcal{H}^t\, .
$$
\end{theorem}
Federer \cite[2.10.24]{federer} proved this result under the restrictive assumption that one of the following two conditions is satisfied:
(a') The function $g$ is positive on a set of $\sigma$-finite measure $\H^t$;  or
(b') the space $Y$ is boundedly compact. Therefore he could only prove Theorem~\ref{T14} under the assumptions (a) or (b) listed above.

While the inequality
$$
\int^*_Y g(y) \ d\mathcal{H}^t(y)\geq \int^\bullet_Y g(y)\, d\H^t(y)
$$
is easy to prove in the general case (see \eqref{eq6}), the problem is to prove the opposite inequality (Federer proved it when (a') or (b') holds true). In the general case,
Federer \cite[p. 187]{federer} stated the following: 

{\em The general problem whether or not the preceding inequality can always be replaced by the corresponding equation is unsolved.}

The problem was answered in the positive by Davies \cite[page~236]{davies}: 

{\em Note added 8 September 1969. H. Federer tells me that this work
answers a question he raised in Geometric measure theory (Berlin, 1969) [...]} 

There is no explicit proof of Theorem~\ref{T15} in the work of Davies, but the main result of Davies \cite[Theorem~8, Example~1]{davies}, provides a missing step in generalizing Federer's proof. In fact it is the celebrated Increasing Sets Lemma \cite[Theorem~8]{davies} that was needed to complete Federer's proof:

\begin{theorem}
\label{T19}
Suppose $(X,d)$ is an arbitrary metric space, $t \in [0,\infty)$, and $\delta > 0$. Then for any increasing sequence of subsets $A_1 \subset A_2 \subset A_3 \subset \cdots$, 
$$
\mathcal{H}^t_\delta \Big(\bigcup_i A_i\Big) = \lim_{i\to\infty} \mathcal{H}^t_\delta (A_i) \, . 
$$
\end{theorem}

With Theorem~\ref{T15} being true for an arbitrary metric space $Y$, Federer's proof of Theorem~\ref{T14} applies to the case of arbitrary metric spaces $X$ and $Y$. 

From what we could dig out from the literature, it would be fair to call Theorem~\ref{T14} the N\"obeling-Szpilrajn-Eilenberg-Federer-Davies inequality.

Surprisingly, it wasn't until 2009 when Reichel \cite{reichel} in his PhD thesis, re-wrote a complete proof of Theorem~\ref{T14} in its full generality, by following the original proof of Federer while making use of Davies' result.
Reichel's thesis seems to be the only place with a complete proof of Theorem~\ref{T14}, except that Reichel did not include the proof of Davies' theorem. 

Davies' theorem \cite[Theorem~8]{davies} (Theorem~\ref{T19} above)
is very difficult and its proof makes use of Ramsey's theorem, ordinal numbers and  non-principal ultrafilters.

In the paper we present a new and elementary proof of Theorem~\ref{T15}
(reformulated below as Theorem~\ref{T5}) that completely avoids the use of Davies' result. It is based on a beautiful argument that we learned from Nazarov \cite{nazarov}. Then we prove Theorem~\ref{T14} including all necessary details. 

Most of the older applications of Theorem~\ref{T14} are in the case of Lipschitz mappings $f:X\to\bbbr^m$ and $t=m$. However, in a recent development of analysis on metric spaces, the general version of Theorem~\ref{T14} plays an increasingly important role. 
It is a fundamental result and it deserves to have a proof that is self-contained and easy to read. Our proof of how to conclude Theorem~\ref{T14} from Theorem~\ref{T15},
follows Federer's argument, but we believe is much easier to read than Federer's proof. In writing this proof we also used a presentation of Federer's proof given in \cite{reichel}. 

In fact we prove more general versions of Theorem~\ref{T14} in Section~\ref{NC}: Theorem~\ref{T20} and Theorem~\ref{T30}. As explained in Remark~\ref{R7} these are substantial improvements of Theorem~\ref{T14}.
In Section~\ref{MC} we show an application of Theorem~\ref{T20}
to the {\em $(n,m)$-mapping content} introduced in \cite{azzams,davids}. 

The paper is structured as follows.
Section~\ref{PR} contains basic material from  measure theory needed in the rest of the paper. This material is standard, but some of the results, although contained in Federer's book, seem to be not very well known. The reader might want to skip Section~\ref{PR}, go directly to Section~\ref{WI} and return to Section~\ref{PR} whenever necessary.

Section~\ref{WI} defines the weighted integrals and weighted measures and proves Lemma \ref{T17} which is a version of Theorem~\ref{T14} with weighted integral in place of the upper integral. This section also has statements of the two main results regarding weighted integral: Theorem~\ref{T7} and Theorem~\ref{T5} (i.e., Theorem~\ref{T15}).

Section~\ref{CT} is focused on Theorem~\ref{T11} and Corollary~\ref{T13} which are of independent interest. These are general results that are essentially combinatorial and are not limited to the specific setting of our problem. They play a central role in the proof of Theorem~\ref{T7}.

In Section~\ref{3.11} we prove Theorem~\ref{T7}. The proof is very short only because of the use of powerful Corollary~\ref{T13}.

In Section~\ref{3.13} we prove Theorem~\ref{T5}. Section \ref{NC} contains the proof of Theorem~\ref{T14} and its generalization Theorem~\ref{T20}. We end with applications to the mapping densities, introduced in \cite{HZ}, and to the $(n,m)$-mapping content \cite{azzams,davids}. Theorem~\ref{T30} can be viewed as yet another coarea inequality, although only under finer assumptions on the metric spaces.

\subsection{Notation}
Open and closed balls in a metric space $(X,d)$ will be denoted by $B(x,r)=\{y:\, d(x,y)<r\}$ and $\bar{B}(x,r)=\{y:\, d(x,y)\leq r\}$, respectively. Closure of a set $E$ will be denoted by $\bar{E}$; as a warning, note that in general closed ball might be strictly larger than the closure of the open ball. Symbol $B$ will always be used to denote a ball, open or closed. If $B=B(x,r)$ is a ball, $\sigma B=B(x,\sigma r)$, $\sigma>0$, will denote a dilated ball (the same notation is used for closed balls).

The characteristic function of a set $E$ will be denoted by $\chi_E$.

A metric space is {\em boundedly compact} if bounded and closed sets are compact.

A  map $f:X \to Y$ between metric spaces is called {\em Lipschitz} if there exists an $L\geq 0$ such that $d_Y (f(x),f(y)) \leq L d_X(x,y)$ for all   $x$  and  $y$  in $X$.
The smallest such $L$, denoted $\lip{f}$, is \textit{the Lipschitz constant} of $f$.

The integral average will be denoted by the barred integral:
$$
\mvint_E f\, d\mu =\frac{1}{\mu(E)}\int_E f\, d\mu.
$$
Hausdorff measure will be denoted by $\H^s$. It is normalized so that on $\bbbr^n$ the measure $\H^n$ coincides with the Lebesgue measure, see Section~\ref{HM} for more details.

For $A \subset X$, $\diam A = \sup\{d(x,y): x,y \in A\}$
and
$$
\zeta^s (A) = \frac{\omega_s}{2^s}(\diam A)^s,
\quad
\text{where}
\quad
\omega_s = \frac{\pi^{s/2}}{\Gamma(\frac{s}{2}+1)}.
$$
Note that $\omega_n$ is the volume of the unit ball in $\bbbr^n$ so $\zeta^n(B^n(0,r))=\H^n(B^n(0,r))$.
Note also that $\zeta^0(A)=1$ if $A\neq\varnothing$ and $\zeta^0(\varnothing)=0$.

For $\delta\in (0,\infty]$, 
a covering $E\subset\bigcup_{i=1}^\infty A_i$ by bounded sets satisfying 
$\diam A_i \leq \delta$ for all $i \in \bbbn$, is called a {\em $\delta$-covering of $E$.} An \emph{open} (\emph{closed}) $\delta$-covering is one where every $A_i$ is open (closed).

\noindent
\textbf{Acknowledgement.} We would like to express our deepest gratitude to Fedor Nazarov for his kindness in providing us with 
an elementary proof of inequality \eqref{eq21}, through 
MathOverflow \cite{nazarov}.
We would also like to thank Mikhail Korobkov for discussions on topics related to Definition~\ref{D1}. Finally, the authors would like to thank the MathOverflow community for providing the reference to N\"obeling's paper \cite{MO2}.

\section{Preliminaries}
\label{PR}

\subsection{Upper Integral}
\label{UI}
Throughout Section~\ref{UI}, $(X,\mu)$ is a measure space.
\begin{definition}
For a function $f : X \to [0,\infty]$ defined $\mu$-a.e. on $X$, the \textit{upper integral} is
defined by
$$
\int^*_X f \ d\mu = \inf \, \int_X \phi \, d\mu \, ,
$$
where the infimum is taken over all \textit{$\mu$-measurable} functions $\phi$ satisfying $0 \leq f(x) \leq \phi(x)$ for $\mu$-a.e. $x \in X$. 
\end{definition}
We do not require $f$ to be measurable. Clearly, for measurable functions the upper integral coincides with the Lebesgue one.
Note also that 
\begin{equation}
\label{eq37}
\text{If $\int_X^*f\, d\mu=0$, then $f=0$, $\mu$-almost everywhere and hence $f$ is measurable.}
\end{equation}
\begin{lemma}
\label{T1}
Let $f_n: X \to [0,\infty]$ be a monotone sequence of (not necessarily measurable) functions, i.e. 
$0 \leq f_1(x) \leq f_2(x) \leq \ldots$ for $\mu$-a.e. $x \in X$. If $f(x):=\lim_{n \to \infty} f_n(x)$, then
\begin{equation}
\label{eq1}
\lim_{n\to \infty} \int_X^* f_n \, d\mu = \int_X^* f \, d\mu \, .
\end{equation}
\end{lemma}
\begin{proof}
Throughout the proof, inequalities between functions are assumed to hold $\mu$-a.e. 
Clearly the limit on the left hand side of \eqref{eq1} exists and
\begin{equation}
\label{eq2}
\lim_{n\to\infty} \int_X^* f_n\, d\mu\leq\int_X^* f\, d\mu.
\end{equation}
Choose measurable functions $\phi_n$ such that $0\leq f_n\leq\phi_n$ and
$$
\int_X\phi_n\, d\mu\leq \int_X^* f_n\, d\mu+2^{-n}.
$$
This and Fatou's lemma yield
$$
\int_X^* f\, d\mu =\int_X^* \lim_{n\to\infty} f_n\, d\mu
\leq
\int_X\liminf_{n\to\infty} \phi_n\, d\mu\leq
\liminf_{n\to\infty} \int_X \phi_n\, d\mu\leq
\lim_{n\to\infty} \int_X^*f_n\, d\mu,
$$
which together with \eqref{eq2} proves \eqref{eq1}.
\end{proof}

\begin{definition}
We say $\phi:X \to [0,\infty]$ is a \textit{step function} if it is $\mu$-measurable and attains at most countably many values (we allow infinite values). That is, $\phi$ is a step function if there exist  disjoint $\mu$-measurable subsets $A_i \subset X $ and $0 < a_i \leq \infty $ such that
\begin{equation}
\label{simple-function}
\phi(x) = \sum_{i=1}^\infty a_i \chi_{A_i}(x) \, .
\end{equation}
\end{definition}
\begin{lemma}
\label{T2}
Let $f:X \to [0,\infty]$ be any function. Then
$$ 
\int^*_X f \ d\mu = \inf \, \int_X \phi \, d\mu \, ,
$$
where the infimum is over all step functions $\phi$ satisfying
$0 \leq f(x) \leq \phi(x)$ for all $x\in X$.
\end{lemma}
\begin{proof}
Since the claim is true when $\int_X^* f \, d\mu = \infty$, we can assume that $\int_X^* f\,d\mu<\infty$.
We can also assume that $f$ is measurable since the general case will easily follow from the definition of the upper integral.
For $i\in\bbbz$ and $1<\lambda<\infty$ define
$$
A_\infty =\{x:\, f(x)=+\infty\},
\quad
\text{and}
\quad
A_i^\lambda=\{x:\, \lambda^i\leq f(x)<\lambda^{i+1}\}.
$$
Then 
$$
f\leq\phi_\lambda\leq\lambda f,
\quad
\text{where}
\quad
\phi_\lambda = \infty\cdot\chi_{A_\infty}+
\sum_{i\in\bbbz} \lambda^{i+1}\chi_{A_i^\lambda}
$$
and
$$
\int_X f\, d\mu\leq \int_X\phi_\lambda\, d\mu\leq \lambda\int_X f\, d\mu\to 
\int_X f\, d\mu
\quad
\text{as $\lambda\to 1^+$}
$$
complete the proof.
\end{proof}

\subsection{Covering lemma}
A familiar $5r$-covering lemma, known also as a Vitali type covering lemma, asserts that from any family $\mathcal{F}$ of balls with bounded radii in a metric space, we can select a subfamily $\mathcal{F}'$ of pairwise disjoint balls such that balls in $\mathcal{F}'$ dilated $5$ times, cover all balls in $\mathcal{F}$, see e.g. \cite[Theorem~3.3]{simon}. A close inspection of the proof reveals that we do not really use the fact that this is a family of balls since the proof is based on  simple estimates for diameters. Therefore, the lemma holds true for any family of uniformly bounded sets, provided we give a proper meaning of being dilated $5$ times. This gives (cf.\ \cite[Section~2.8]{federer})
\begin{lemma}
\label{T9}
Let $\mathcal{F}$ be a family of bounded sets in a metric space such that
$\sup\{\diam F:\, F\in\mathcal{F}\}<\infty$. Then, there is a subfamily $\mathcal{F}'\subset\mathcal{F}$ of pairwise disjoint sets such that
$$
\bigcup_{F\in\mathcal{F}} F\subset \bigcup_{F'\in \mathcal{F}'}\5 F',
$$
where
$$
\5 F'=\bigcup\{F\in\mathcal{F}:\, F\cap F'\neq\varnothing,\ \diam F\leq 2\diam F'\}.
$$
Moreover, if $F\in \mathcal{F}$, then there is $F'\in \mathcal{F}'$ such that $F\cap F'\neq\varnothing$ and $F\subset\5 F'$.
\end{lemma}
\begin{remark}
That is $\5 F'$ is the union of $F'$ and all sets that intersect it and have relative small diameter. 
Clearly $\diam \5 F'\leq 5\diam F'$. 
\end{remark}
\begin{proof}
Let $\sup\{\diam F:\, F\in\mathcal{F}\}=R<\infty$ and let
$$
\mathcal{F}_j=\left\{F\in\mathcal{F}:\, \frac{R}{2^j}<\diam F\leq \frac{R}{2^{j-1}}\right\} \, .
$$
So, $ \bigcup_{j=1}^\infty \mathcal{F}_j $ includes all of $ \mathcal{F} $ except possibly for some singletons -- sets of diameter zero.

We define $\mathcal{F}_1'\subset\mathcal{F}_1$ to be a maximal family of pairwise disjoint sets in $\mathcal{F}_1$. Suppose that the families $\mathcal{F}_1',\ldots,\mathcal{F}_{j-1}'$ have already been defined. Then we define $\mathcal{F}_j'$ to be a maximal family of pairwise disjoint sets in
$$
\{F \in \mathcal{F}_j: F\cap F'=\varnothing
\text{ for all $F'\in \mathcal{F}_1'\cup\ldots\cup\mathcal{F}'_{j-1}$}\}\, .
$$
Set $ \mathcal{F}'=\bigcup_{j=1}^\infty \mathcal{F}_j'.$
Every set $F\in\mathcal{F}_j$ intersects with a set $F'\in\bigcup_{i=1}^j \mathcal{F}_i'$;
it follows that $\diam F\leq 2\diam F'$ and hence
$F\subset \5 F'$.

If there are any singletons $F=\{x\} \in \mathcal{F}$ such that $x \notin \bigcup_{F'\in \mathcal{F}'}F'$, then add $F$ to the collection $\mathcal{F}'$. The updated $\mathcal{F}'$ will remain disjointed and now it satisfies the claim of the lemma.
\end{proof}

\begin{definition}
Let $\mathcal{F}$ be a family of sets in a metric space $X$. We say that the family $\mathcal{F}$ is a {\em fine covering} of a set $A\subset X$ if for every $x\in A$ and every $\eps>0$, there is $F\in\mathcal{F}$ such that $x\in F\subset B(x,\eps)$. 

\end{definition}
\begin{corollary}
\label{T8}
If $\mathcal{F}$ is a family of {\em closed} sets that forms a fine covering of $A\subset X$,
$\sup\{\diam F:\, F\in\mathcal{F}\}<\infty$, and $\mathcal{F}'$ is as in Lemma~\ref{T9}, then for any finite collection of sets $F_1',\ldots,F_N'\in\mathcal{F}'$ we have
\begin{equation}
\label{eq13}
A\subset\bigcup_{j=1}^N F_j'\cup\bigcup_{F'\in \mathcal{F}'\setminus \{F_1',\ldots,F_N'\}} \5 F'
\end{equation}
\end{corollary}
\begin{proof}
If $x\in A\setminus\bigcup_{j=1}^N F_j'$, since the sets $F_j'$ are closed, a ball $B(x,\eps)$ is disjoint with the sets $F_j'$. If $x\in F\subset B(x,\eps)$, $F\in\mathcal{F}$, then there is $F'\in\mathcal{F}'$ such that $F\cap F'\neq\varnothing$ and $x\in F\subset\5 F'$.
Since $F\subset B(x,\eps)$ and $B(x,\eps)\cap F_j'=\varnothing$, $F'\neq F_j'$ and hence 
$F'$ is one of the sets on the right hand side of \eqref{eq13}.

\end{proof}

\subsection{Hausdorff Measures}
\label{HM}
Let $(X,d)$ be a metric space. Fix an $0 \leq s < \infty $. For a subset $E$ of $X$ and a $\delta \in (0,\infty]$, the {\em Hausdorff contents} $\H^s_\delta$ and $\mathscr{H}^s_\delta$ are defined by
$$
\H^s_\delta (E)
= \inf \sum_{i=1}^\infty \zeta^s (A_i), 
\quad
\text{and}
\quad
\mathscr{H}^s_\delta (E)
= \inf \sum_{i=1}^\infty \zeta^s (U_i)
$$
where the infima are taken, respectively,  over all countable coverings $E\subset\bigcup_{i=1}^\infty A_i$ by bounded sets with 
$\diam A_i \leq \delta$ for all $i \in \bbbn$, and over all countable coverings $E\subset\bigcup_{i=1}^\infty U_i$  by \emph{open} sets with $\diam U_i\leq\delta$ for all $i \in \bbbn$, in other words, over all $\delta$-coverings and over all \textit{open} $\delta$-coverings.
If no such covering(s) exists, we set the corresponding content equal to $+\infty$. 

Note that we can always assume that the sets $A_i$ are closed since taking the closure of a set does not increase its diameter. Note also that for any $0<\eps<\delta < \infty$
$$
\H_\delta^s(E)\leq\mathscr{H}_\delta^s(E)\leq\H^s_{\delta-\eps}(E),
$$
because any $(\delta-\eps)$-covering can be enlarged to an open $\delta$-covering with an arbitrarily small increase in diameters of the sets.

The functions $\delta\mapsto\H^s_\delta(E)$ and $\delta\mapsto\mathscr{H}^s_\delta(E)$ are non-increasing, hence for $0\leq s<\infty$
$$
\H^s(E) := \lim_{\delta \to 0^+} \H^s_\delta (E) = \sup_{\delta > 0} \H^s_\delta (E)= 
\lim_{\delta \to 0^+} \mathscr{H}^s_\delta (E) = \sup_{\delta > 0} \mathscr{H}^s_\delta (E)\, ,
$$
is well-defined.  This is the \textit{$s$-dimensional Hausdorff measure} on $X$.

Note that $\H^0$ is the {\em counting measure}, i.e. $\H^0(E)$ equals the number of elements of $E$.

The Hausdorff measure is an outer measure defined on all subsets of $X$ and all Borel sets are $\H^s$-measurable.
\begin{remark}
If $n \in \bbbn$, then $\omega_n$ equals the volume of the unit ball in $\bbbr^n$. With this choice of the normalizing coefficient,
$\H^n=\H^n_\infty=\mathcal{L}^n$ in $\bbbr^n$, where $\mathcal{L}^n$ is the outer Lebesgue measure,
see \cite[Theorem~2.6]{simon}.
However, we will not use this fact in what follows. 
\end{remark}
The next result proves that the Hausdorff measure is Borel-regular.
\begin{lemma}
\label{T33}
For $s\in [0,\infty)$ and
every $E\subset X$ there is a 
decreasing sequence of open sets $V_1\supset V_2\supset\ldots\supset E$ such that
$E\subset\tilde{E}:=\bigcap_{i=1}^\infty V_i$ and $\H^s(E)=\H^s(\tilde{E})$. 
\end{lemma}
\begin{proof}
If $\H^s(E)=\infty$ then we can take $V_i=X$, for all $i \in \bbbn$. So, assume $\H^s(E) < \infty$. 
For each $i\in \bbbn$ there is a $1/i$-covering $E\subset\bigcup_{j=1}^\infty U_{ij}:=U_i$ by open sets, such that
$$
\sum_{j=1}^\infty \zeta^s(U_{ij})\leq \mathscr{H}^s_{1/i}(E)+\frac{1}{i}
\quad
\text{so}
\quad
\mathscr{H}_{1/i}^s(U_i)\leq \H^s(E)+\frac{1}{i} \, .
$$
Let $V_i=\bigcap_{k=1}^i U_k$,
then $\tilde{E}=\bigcap_{i=1}^\infty U_i=\bigcap_{i=1}^\infty V_i$ has the required properties.
\end{proof}
As an immediate consequence we get 
\begin{lemma}
\label{T10}
If $0 \leq s < \infty $, $\H^s(X)<\infty$ and $E\subset X$ is any set, then
\begin{equation}
\label{eq15}
\H^s(E)=\inf\{\H^s(U):\, \text{$ U \supset E$, $U$ is open}\}.
\end{equation}
\end{lemma}
The next result is slightly less obvious
\begin{lemma}
\label{T36}
Let $E\subset X$ be any $\H^s$-measurable set, $0 \leq s<\infty$. If $\H^s(E)<\infty$ then
$$
\H^s(E)=\sup\{\H^s(C):\, \text{$C\subset E$, $C$ is closed}\}.
$$
\end{lemma}
\begin{proof}
It is enough to prove that for any $\eps>0$ there exists an $F_\sigma$-set contained in $E$ with $\H^s$-measure larger than $\H^s(E)-\eps$.

Fix $\eps > 0$. Let $\tilde{E}=\bigcap_{i=1}^\infty V_i$, $\H^s(\tilde{E})=\H^s(E)$ be the $G_\delta$ set from Lemma~\ref{T33}. Since $E$ is measurable and has finite measure, $\H^s(\tilde{E}\setminus E)=0$. Each of the open sets $V_i$ is a union of an increasing sequence of closed sets. Since $E$ is contained in that union, there is a closed set $F_i\subset V_i$ such that
$\H^s(E\setminus F_i)<\eps/2^{i}$ and hence the closed set
$F=\bigcap_{i=1}^\infty F_i \subset \bigcap_{i=1}^\infty V_i = \tilde{E}$ satisfies
$$
\H^s(E\setminus F)=\H^s\Big(\bigcup_{i=1}^\infty (E\setminus F_i)\Big)<{\eps}.
$$
Since $\H^s(F\setminus E) \leq \H^s(\tilde{E}\setminus E)=0$, by Lemma~\ref{T33}, there exits a $G_\delta$-set $G$ such that $F\setminus E\subset G$ and $\H^s(G)=0$. Now $F \setminus G$ is an $F_\sigma$-set contained in $E$ and 
$$
\H^s(F \setminus G) = \H^s(F)\geq \H^s(E) - \H^s(E\setminus F) > \H^s(E) - \eps   \, .
$$
\end{proof}
\begin{lemma}
\label{T34}
If $s\in [0,\infty)$ and $A_1\subset A_2\subset\ldots$ is an increasing sequence of (not necessarily measurable) sets, then
\begin{equation}
\label{eq38}
\H^s\Big(\bigcup_{i=1}^\infty A_i\Big)=\lim_{i\to\infty} \H^s(A_i).
\end{equation}
\end{lemma}
\begin{proof}
It suffices to prove that the right hand side of \eqref{eq38} is greater than or equal to the left hand side; the opposite inequality is obvious. Let $\hat{A}_i$ be a Borel set such that $A_i\subset\hat{A}_i$, and  $\H^s(A_i)=\H^s(\hat{A}_i)$. Let 
$\tilde{A}_i=\bigcap_{j=i}^\infty\hat{A}_i$. Then $\tilde{A}_i$ is Borel, $A_i\subset\tilde{A}_i$ and $\H^s(A_i)=\H^s(\tilde{A}_i)$. Since $\tilde{A}_1\subset\tilde{A}_2\subset\ldots$ are measurable, we have
$$
\H^s\Big(\bigcup_{i=1}^\infty A_i\Big)\leq\H^s\Big(\bigcup_{i=1}^\infty\tilde{A}_i\Big)=
\lim_{i\to\infty} \H^s(\tilde{A}_i)=\lim_{i\to\infty} \H^s(A_i).
$$
\end{proof}

If a set $F$ is bounded, then $\H^s_\infty(F)\leq\zeta^s(F)$ is an obvious estimate. However, in general we may expect that $\H^s(F)$ is much larger than $\zeta^s(F)$. Indeed, sets with small diameters may have arbitrarily large Hausdorff measure. There is no need to convince the reader that life would be much easier if we could estimate $\H^s(F)$ in terms of the diameter, say $\H^s(F)\leq (1+\eps)\zeta^s(F)$ for some small $\eps$. The next result shows that in fact, in spaces of finite measure, at {\em almost all} locations and {\em all} small scales this estimate is true.
\begin{lemma}
\label{T6}
Let $0\leq s<\infty$ and $\eps>0$. If $\H^s(X)<\infty$, then there is a set $E\subset X$ of measure zero, $\H^s(E)=0$, such that
\begin{equation}
\label{eq14}
\forall\, x\in X\setminus E\ \  
\exists\ \delta_x>0\ \ 
\forall\ F\subset X\ \ \ 
\left(x\in F\subset \bar{B}(x,\delta_x)\ 
\Rightarrow
\ \H^s(F)\leq (1+\eps)\zeta^s(F)\right).
\end{equation}
\end{lemma}
\begin{remark}
We do not assume measurability of the sets $F$.
\end{remark}
\begin{proof}
The claim is obvious for $s=0$, so assume $s>0$. Since $\zeta^s(F)=\zeta^s(\bar{F})$, it suffices to prove \eqref{eq14} for closed sets $F$. Let $E\subset X$ be the set of all points $x\in X$ such that for every $j\in\bbbn$, there is a closed set $F_{x,j}$ satisfying 
$$
x\in F_{x,j}\subset\bar{B}(x,1/j)
\quad
\text{and}
\quad
\H^s(F_{x,j})>(1+\eps)\zeta^s(F_{x,j}).
$$
Clearly, with this definition of $E$, \eqref{eq14} is true and it remains to show that $\H^s(E)=0$.
Suppose to the contrary $\H^s(E)>0$. According to Lemma~\ref{T10}, there is an open set $U$ such that $E\subset U$ and $\H^s(U)<\H^s(E)(1+\eps/4)$.
Given $\delta>0$, the family 
$$
\mathcal{F}=\{F_{x,j}: F_{x,j}\subset U,\ j\geq 10/\delta,\ x\in E\}
$$
is a fine covering of $E$ by closed sets. 
Note that $F_{x,j}\subset \bar{B}(x,1/j)$, $\diam F_{x,j}\leq 2/j\leq\delta/5$.
Lemma~\ref{T9} yields $ \mathcal{F}'\subset \mathcal{F}$ such that
$$
E\subset\bigcup_{F'\in\mathcal{F}'} \5 F',
$$
and the closed sets $F'\in\mathcal{F}'$ are pairwise disjoint. Since $\H^s(X)<\infty$, only countably many of them may have positive measure and the sum of measures is finite so there is a finite collections of sets
$F_1',\ldots,F_N'\in\mathcal{F}'$ such that 
$$
\sum_{F'\in\mathcal{F}'\setminus\{F_1',\ldots,F_N'\}} \H^s(F')<5^{-s}\H^s(E)\frac{\eps}{4}\, .
$$
According to Corollary~\ref{T8}, 
$$
E\subset\bigcup_{j=1}^N F_j'\; \cup\bigcup_{F'\in\mathcal{F}'\setminus\{F_1',\ldots,F_N'\}}
\5 F'.
$$
Since for each of the sets $F'\in\mathcal{F}'$ we have, $F'\subset U$, $\diam\5 F'\leq 5\diam F'\leq \delta$, 
\begin{equation*}
\begin{split}
\H^s_{\delta} (E) 
&\leq    
\sum_{j=1}^N \zeta^s(F_j') + 
\sum_{F'\in\mathcal{F}'\setminus\{F_1',\ldots,F_N'\}} \zeta^s(\5 F')\\
&\leq
\sum_{j=1}^N \zeta^s(F_j') + 
\sum_{F'\in\mathcal{F}'\setminus\{F_1',\ldots,F_N'\}} 5^s\zeta^s(F')\\
&\leq
\frac{1}{1+\eps}\left(\sum_{j=1}^N \H^s(F_j') + 
\sum_{F'\in\mathcal{F}'\setminus\{F_1',\ldots,F_N'\}} 5^s\H^s(F')\right)\\
&\leq
\frac{1}{1+\eps}\left(\H^s(U)+\H^s(E)\frac{\eps}{4}\right)
\leq
\H^s(E)\frac{1+\eps/2}{1+\eps}.
\end{split}    
\end{equation*}
The estimate is independent of $\delta$ so letting $\delta\to 0^+$ we get
$$
\H^s(E)\leq \H^s(E)\frac{1+\eps/2}{1+\eps}<\H^s(E)
$$ 
which is a clear contradiction.
\end{proof}

\section{Weighted Integral and weighted Haudorff measure}
\label{WI}
Throughout this section $(X,d)$ will be a metric space and functions $f:X\to [0,\infty]$ will not necessarily be measurable.
\begin{definition}
\label{D3}
For a function $f:X\to [0,\infty]$, a {\em weighted covering of $f$} is a
countable collection $\{(a_i,A_i)\}_{i \in \bbbn}$ of pairs of bounded sets $A_i \subset X$ and numbers $a_i \in [0,\infty]$ such that 
\begin{equation}
\label{eq3}
f(x) \leq \sum _i a_i \chi_{A_i}(x) \quad \text{for all} \quad x \in X \, .
\end{equation}
If in addition $\diam A_i\leq\delta$, $\delta\in(0,+\infty]$, for all $i\in\bbbn$, we say that $\{(a_i,A_i)\}_{i \in \bbbn}$ is a
{\em weighted $\delta$-covering of $f$}.
If $f=\chi_E$ we call $\{(a_i,A_i)\}_{i\in\bbbn}$ a {\em  weighted ($\delta$-)covering of $E$.}

Let $\delta \in (0,+\infty]$, and $s \in [0,\infty)$. The {\em weighted integral of $f$} is defined by
\begin{equation}
\label{eq4}
\int_X^\bullet f \, d\H^s_\delta := 
\inf \sum_{i=1}^\infty a_i \zeta^s(A_i),
\end{equation}
where the infimum is taken over all weighted $\delta$-coverings of $f$, and
$$
\int_X^\bullet f\, d\H^s=\lim_{\delta\to 0^+} \int_X^\bullet f\, d\H_\delta^s.
$$
Note that the limit exists since the integral \eqref{eq4} is non-increasing in $\delta$.

If no $\delta$-cover of $f$ exists, we set the weighted integral of $f$ to be  $+\infty$.
\end{definition}
\begin{remark}
\label{R1}
Since the diameter of a set and of its closure are equal, we may assume that the sets $A_i$ are closed.
\end{remark}

\begin{definition}
The {\em weighted Hausdorff content} 
and the {\em weighted Hausdorff measure}
of a set $E\subset X$ are respectively defined by
$$
\lambda^s_\delta(E)=\int_X^\bullet \chi_E\, d\H^s_\delta
\quad
\text{and}
\quad
\lambda^s(E)=\lim_{\delta\to 0^+}\lambda_\delta^s(E)=\int_X^\bullet \chi_E\, d\H^s.
$$
In other words
$\lambda_\delta^s(E)=\inf\sum_{i=1}^\infty a_i\zeta^s(A_i)$,
where the infimum is over all collections $\{(a_i,A_i)\}_{i\in\bbbn}$ such that
$\sum a_i\chi_{A_i}(x) \geq 1$ for all $x\in E$, and $\diam A_i\leq\delta$, for all $i\in\bbbn$.
\end{definition}
\begin{remark}
Note that while in the definition of a step function we assumed that the sets $A_i$ were disjoint, the sets $A_i$ here are not required to be disjoint. A step function uniquely determines the sets $A_i$ and numbers $a_i$, but the same function on the right hand side of 
\eqref{eq3} can be represented in several different ways. 
It is important that the infimum in \eqref{eq4} is taken over all collections $\{(a_i,A_i)\}$ and not only over those corresponding to step functions.
\end{remark}
\begin{remark}
\label{R4}
It seems that Federer \cite[2.10.24]{federer} was the first to define weighted integrals. He denoted them by $\lambda_\delta(f)$ but did not use any terms to refer to them.
The first systematic study of weighted measures was done by
Kelly \cite{kelly1,kelly2} under the name of {\em method III measures}, although he is using the name weighted covering.
The name weighted Hausdorff measures was introduced by Howroyd \cite{howroyd}, see also \cite[Chapter~8]{mattila}. The term weighted integral and the notation $\int_X^\bullet f\, d\H^s_\delta$ appears in \cite{reichel}.
\end{remark}

\subsection{Coarea inequality for weighted integrals}
To provide motivation for the notion of the weighted integral, we will prove \eqref{eq24}. In fact we will prove a slightly more general inequality that applies to any uniformly continuous map between metric spaces. The point is that the notion of weighted integral is designed to make the proof very easy.

\begin{definition}
\label{D1}
For an arbitrary map $f:X\to Y$ between metric spaces, $s,t\in [0,\infty)$, $\delta\in (0,\infty]$, and any $E\subset X$ we define 
$$
\Phi^{s,t}_\delta(f,E):=\inf\sum_{i=1}^\infty \zeta^s(f(A_i))\zeta^t(A_i),
$$
where the infimum is taken over all $\delta$-coverings $\{A_i\}_{i=1}^\infty$ of $E$. Obviously, $\delta\mapsto\Phi_\delta^{s,t}$ is non-increasing, allowing the definition
$$
\Phi^{s,t}(f,E):=\lim_{\delta\to 0^+}\Phi^{s,t}_\delta(f,E).
$$
\end{definition}
\begin{remark}
This definition is motivated by a similar definition in \cite[Appendix~A]{HKK} and also by the definition of the mapping content introduced in \cite{azzams,davids}, see Definition~\ref{bD1}.
\end{remark}
The proofs of 
the next two easy results are left to the reader.
\begin{lemma}
\label{T35}
For any $\delta\in (0,\infty]$, $s,t\in [0,\infty)$, $E,F\subset X$, and $f:X\to Y$ we have
$$
\Phi_\delta^{s,t}(f,E\cup F)\leq \Phi_\delta^{s,t}(f,E)+\Phi_\delta^{s,t}(f,F)
\quad
\text{so}
\quad
\Phi^{s,t}(f,E\cup F)\leq \Phi^{s,t}(f,E)+\Phi^{s,t}(f,F).
$$
\end{lemma}
\begin{lemma}
\label{T16}
If $f:X\to Y$ is Lipschitz continuous and $E\subset X$, $s,t \in [0,\infty)$, and $\delta\in (0,\infty]$, then
$$
\Phi^{s,t}_\delta(f,E)\leq (\lip f)^s\, \frac{\omega_s\omega_t}{\omega_{s+t}}\, \H_\delta^{s+t}(E)
\quad
\text{so}
\quad
\Phi^{s,t}(f,E)\leq (\lip f)^s\, \frac{\omega_s\omega_t}{\omega_{s+t}}\, \H^{s+t}(E).
$$
\end{lemma}
The next version of the coarea inequality easily follows from the definition of the weighted integral and is a building block of the proof of the main coarea inequality, Theorem~\ref{T20}.
\begin{lemma}
\label{T17}
If $f:X\to Y$ is a uniformly continuous map between metric spaces, $0 \leq t \leq s<\infty$ and $E\subset X$, then
\begin{equation}
\label{eq25}
\lim_{\delta\to 0^+}\int_Y^\bullet \H^{s-t}_\delta(f^{-1}(y)\cap E)\, d\H^t(y)\leq \Phi^{t,s-t}(f,E).
\end{equation}
\end{lemma}
\begin{remark}
At this point it is not entirely clear that we can pass to the limit under the sign of the integral as $\delta\to 0^+$, since we do not a priori have the monotone convergence theorem for weighted integrals. In fact such a result is true since according to Theorem~\ref{T5}, the weighted integral equals the upper integral, but Theorem~\ref{T5} is difficult.
\end{remark}
\begin{remark}
\label{R5}
Lemma~\ref{T17} and Lemma~\ref{T16} yield that if in addition $f$ is Lipschitz continuous, then
$$
\lim_{\delta\to 0^+}\int_Y^\bullet \H^{s-t}_\delta(f^{-1}(y)\cap E)\, d\H^t(y)\leq
\Phi^{t,s-t}(f,E)\leq (\lip f)^t\, \frac{\omega_{s-t}\omega_t}{\omega_s}\, \H^s(E).
$$
Therefore Theorem~\ref{T14} is a consequence of easy Lemma~\ref{T17} and the deep Theorem~\ref{T15} (reformulated below as Theorem~\ref{T5}).
\end{remark}
\begin{proof}[Proof of Lemma~\ref{T17}]
 Assume that $\Phi^{t,s-t}(f,E)<\infty$, as otherwise the inequality is obvious. Fix $ \delta_o \in (0,\infty] $. Given $\eps>0$ and $0<\delta \leq \delta_o$, let $\{ A_i\}_{i=1}^\infty$ be a $\delta$-covering of $E$ such that
\begin{equation}
\label{eq26}
\sum_{i=1}^\infty \zeta^t(f(A_i))\zeta^{s-t}(A_i)<\Phi_\delta^{t,s-t}(f,E)+\eps.
\end{equation}
Since the sets $\{A_i:\, y\in f(A_i)\}$ form a $\delta$-covering of $f^{-1}(y)\cap E$, we have 
\begin{equation}
\label{eq27}
\H^{s-t}_{\delta_o}(f^{-1}(y)\cap E)\leq \sum_{i=1}^\infty a_i\chi_{F_i}(y),
\quad
\text{where}
\quad
\text{$a_i=\zeta^{s-t}(A_i)$ and $F_i=f(A_i)$.}
\end{equation}
Since the mapping $f$ is uniformly continuous,
$$
\eta(\delta)=\sup_{A\subset X\atop\diam A\leq\delta} \diam f(A)\to 0
\quad
\text{as $\delta\to 0^+$}.
$$
According to \eqref{eq27}, $\{(a_i,F_i)\}_{i=1}^\infty$ forms a weighted $\eta(\delta)$-covering of the function 
$y\mapsto\H^{s-t}_{\delta_o}(f^{-1}(y)\cap E)$ and the definition of the weighted integral yields
$$
\int_Y^\bullet \H^{s-t}_{\delta_o}(f^{-1}(y)\cap E)\, d\H^t_{\eta(\delta)}(y)\leq
\sum_{i=1}^\infty a_i\zeta^t(F_i)<\Phi_\delta^{t,s-t}(f,E)+\eps,
$$
where the last inequality is nothing else, but inequality \eqref{eq26}. Letting $\delta\to 0^+$ first and then $\eps\to 0^+$ proves 
$$
\int_Y^\bullet \H^{s-t}_{\delta_o}(f^{-1}(y)\cap E)\, d\H^t(y)\leq \Phi^{t,s-t}(f,E).
$$
Since, $\delta_o$ was arbitrary, \eqref{eq25} follows.
\end{proof}
The strategy to prove the coarea inequality, Theorem~\ref{T20} (generalizarion of Theorem~\ref{T14}), from Lemma~\ref{T17} is to apply Theorem~\ref{T5} below to replace the weighted integral in inequality \eqref{eq25} with the upper integral and then apply the monotone convergence theorem, Lemma~\ref{T1}. So, it is clear that the heart of the proof lies in proving Theorem~\ref{T5}, and this is the focus of the Sections~4, 5, and~6. Here is where we deviate from literature significantly and provide a new proof that avoids Davies' result, Theorem~\ref{T19}.

\subsection{Fundamental properties of weighted integrals}
Fundamental properties of the weighted Hausdorff measures and the weighted integrals are stated in Theorem~\ref{T7} and Theorem~\ref{T5} which is a reformulation of Theorem~\ref{T15} in a notation consistent with that of Theorem~\ref{T7}.

\begin{theorem}
\label{T7}
Let $X$ be a metric space and $s \in [0,\infty)$. 
Then for any $E\subset X$, 
\begin{equation}
\label{eq22}
\lambda^s(E)=\H^s(E).
\end{equation}
Moreover, if $\delta\in (0,\infty]$, then
\begin{equation}
\label{eq21}
(8\cdot 6^{s})^{-1}\H^s_{6\delta}(E)\leq\lambda^s_\delta(E)\leq \H^s_\delta(E).
\end{equation}
\end{theorem}

\begin{remark}
\label{R3}
Passing to the limit in \eqref{eq21} as $\delta\to 0^+$, yields
$(8\cdot 6^{s})^{-1}\H^s(E)\leq\lambda^s(E)\leq \H^s(E)$
which is weaker than \eqref{eq22} so \eqref{eq22} is somewhat surprising.
\end{remark}
Theorem~\ref{T7} will play a crucial role in the proof of
\begin{theorem}
 \label{T5}
Let $X$ be a metric space.
 For $s \in [0,\infty)$, and any $f:X \to [0,\infty]$ we have
\begin{equation}
\label{eq8}
\int^\bullet_X f \ d\mathcal{H}^s = \int^*_X f \ d\mathcal{H}^s \, .
\end{equation}
\end{theorem}

\begin{remark}
Inequality \eqref{eq21} is stated implicitly in \cite[2.10.24]{federer}, as a step in the proof of Theorem~\ref{T5} (under assumptions (a') or (b')) and the general case follows from the theorem of Davies \cite{davies}, see \cite{howroyd,kelly1,kelly2}.
\end{remark}

\section{Weighted covering theorem}
\label{CT}
The proof of inequality \eqref{eq21} is based on the following weighted covering result that we learned from Nazarov through  MathOverflow \cite{nazarov}. The result is interesting on its own and we believe it will have applications beyond those given in the paper.
\begin{theorem}
\label{T11}
Let $E$ be a bounded and non-empty subset of a metric space. If 
$0\leq b_i<\infty$, $i=1,2,\ldots, N$, are fixed numbers and 
$\{(a_i,B_i)\}_{i=1}^N$ is a finite weighted covering of $E$ by (either all open or all closed) balls i.e.,
\begin{equation}
\label{eq16}
\chi_{E}\leq\sum_{i=1}^N a_i\chi_{B_i},
\quad
a_i\geq 0,
\end{equation}
then there is a subfamily of pairwise disjoint balls $\{B_{i_j}\}_{j=1}^k$ such that
$$
E\subset\bigcup_{j=1}^k 3B_{i_j}
\quad
\text{and}
\quad
\sum_{j=1}^k b_{i_j}\leq 2\sum_{i=1}^N a_ib_i.
$$
\end{theorem}
\begin{remark}
Later, we will apply Theorem~\ref{T11} with $b_i=\zeta^s(B_i)$.
\end{remark}
\begin{proof}
We will prove the result using induction with respect to $N$. More precisely, we will prove that for every $N\in\bbbn$, the statement is true for any set $E$ that is bounded and non-empty and any weighted covering of it with $N$ balls.

It is important to prove the statement for all sets $E$. Proving it for a fixed set $E$ would not work, since the induction hypothesis will be applied to sets different than $E$. Namely, it will be applied to subsets of $E$.

If $N=1$, the claim is obvious, because we have one ball $B_1$ and $a_1\geq 1$. Suppose $N\geq 2$ and the claim is true if the number of balls is less than or equal to $N-1$, we will prove it for $N$ balls.

Let $\{(a_i,B_i)\}_{i=1}^N$ be a weighted covering of $E$ satisfying \eqref{eq16}. 
For $\alpha=(\alpha_1,\ldots,\alpha_N)$, let
$$
W=\Big\{\alpha:\, \alpha_i\geq 0,\ \sum_{i=1}^N \alpha_i\chi_{B_i}\geq \chi_E\Big\},
\quad 
W_c=\Big\{\alpha:\, 1\geq \alpha_i\geq 0,\ \sum_{i=1}^N \alpha_i\chi_{B_i}\geq \chi_E\Big\}.
$$
Let $\psi(\alpha)=\sum_{i=1}^N \alpha_ib_i$. If $\alpha\in W$, then
$$
\alpha\wedge 1=(\min\{\alpha_1,1\},\ldots,\min\{\alpha_N,1\})\in W_c
\quad
\text{and}
\quad
\psi(\alpha\wedge 1)\leq\psi(\alpha)
$$
so $\inf_W\psi=\inf_{W_c}\psi$. Since $W_c$ is compact and non-empty, there is $\alpha\in W_c$ such that
$\psi(\alpha)=\inf_{W_c}\psi=\inf_W\psi$. In particular,
\begin{equation}
\label{eq17}
\sum_{i=1}^N\alpha_ib_i\leq \sum_{i=1}^N a_ib_i.
\end{equation}
If there is $i_o$ such that $\alpha_{i_o}=0$, we are done. Indeed,
$$
\chi_E\leq\sum_{1\leq i\leq N\atop i\neq i_o} \alpha_i\chi_{B_i}
$$
is a weighted covering of $E$ by $N-1$ balls so according to the induction hypothesis, there is a subfamily of pairwise disjoint balls $\{ B_{i_j}\}_{j=1}^k$ such that
$$
E\subset\bigcup_{j=1}^k 3B_{i_j},
\quad
\sum_{j=1}^kb_{i_j}\leq 2\sum_{1\leq i\leq N\atop i\neq i_o} \alpha_i b_i=
2\sum_{i=1}^N \alpha_i b_i \leq 2\sum_{i=1}^N a_i b_i \, .
$$
Therefore, we may assume that $\alpha_i>0$ for all $i\in \{1,\ldots,N\}$.
\begin{lemma}
\label{T12}
If $\alpha\in W_c$ is a minimizer of $\psi$ and $\alpha_i >0$ for all $i$, then for any $i_1\in\{1,\ldots,N\}$, we have
$$
\sum_{\{i:\, B_i\cap B_{i_1}\neq\varnothing\}} 
\alpha_ib_i
\geq\frac{b_{i_1}}{2}\, .
$$
\end{lemma}
\begin{proof}
Since the sum on the left hand side includes $\alpha_{i_1}b_{i_1}$, the claim is obvious if
$\alpha_{i_1}\geq 1/2$. Therefore, we may assume that $0<\alpha_{i_1}<1/2$. Let $0<h<\alpha_{i_1}$ and define
$$
\tilde{\alpha}_i=
\begin{cases}
\alpha_i & \text{if $B_i\cap B_{i_1}=\varnothing$,}\\
\alpha_i(1+2h) & \text{if $B_i\cap B_{i_1}\neq\varnothing$, $i\neq i_1$,}\\
\alpha_i-h & \text{if $i=i_1$}.
\end{cases}
$$
We claim that
\begin{equation}
\label{eq18}
(\tilde{\alpha}_1,\ldots,\tilde{\alpha}_N)\in W
\quad
\text{i.e.,}
\quad
\sum_{i=1}^N\tilde{\alpha}_i\chi_{B_i}\geq \chi_E.
\end{equation}
If $x\not\in B_{i_1}$, then $\tilde{\alpha}_{i_1}\chi_{B_{i_1}}(x)=\alpha_{i_1}\chi_{B_{i_1}}(x)=0$. Since $\tilde{\alpha}_i\geq \alpha_i$ for all $i\neq i_{1}$, we have
\begin{equation}
\label{eq19}
\sum_{i=1}^N \tilde{\alpha}_i\chi_{B_i}(x)\geq\sum_{i=1}^N\alpha_i\chi_{B_i}(x)\geq\chi_E(x).
\end{equation}
If $x\not\in E$, then $\chi_E(x)=0$ and there is nothing to prove. 

If $x\in E\cap B_{i_1}$, then
$$
1=\chi_E(x)\leq \sum_{i=1}^N \alpha_i\chi_{B_i}(x)=
\alpha_{i_1}+\sum_{\{i:\, i\neq i_1,\ x\in B_i\cap B_{i_1}\}} \alpha_i,
$$
and hence
$$
\sum_{\{i:\, i\neq i_1,\ x\in B_i\cap B_{i_1}\}} \alpha_i\geq 
1-\alpha_{i_1}.
$$
Therefore,
\begin{equation*}
\begin{split}
&\sum_{i=1}^N\tilde{\alpha}_i\chi_{B_i}(x)
=
(\alpha_{i_1}-h)+
\sum_{\{i:\, i\neq i_1,\ x\in B_i\cap B_{i_1}\}}\alpha_i(1+2h)\\
&\geq
(\alpha_{i_1}-h)+(1+2h)(1-\alpha_{i_1})
=
1+h(1-2\alpha_{i_1})>1=\chi_E(x),
\end{split}    
\end{equation*}
where the last inequality is a consequence of $0<\alpha_{i_1}<1/2$. This completes the proof of \eqref{eq18}.

Since $\psi$ attains minimum at $\alpha$, we have
\begin{equation}
\label{eq20}
\sum_{i=1}^N\alpha_ib_i\leq \sum_{i=1}^N\tilde{\alpha}_ib_i.
\end{equation}
Since $\alpha_i=\tilde{\alpha}_i$ if $B_i\cap B_{i_1}=\varnothing$, \eqref{eq20} yields
\begin{equation*}
\begin{split}
&\alpha_{i_1}b_{i_1}
+
\sum_{\{i:\, i\neq i_1,\ B_i\cap B_{i_1}\neq\varnothing\}}
\alpha_ib_i\\
&\leq
(\alpha_{i_1}-h)b_{i_1}+\sum_{\{i:\, i\neq i_1,\ B_i\cap B_{i_1}\neq\varnothing\}}
\alpha_i(1+2h)b_i,
\end{split}
\end{equation*}
and hence
$$
h b_{i_1}\leq 2h
\sum_{\{i:\, i\neq i_1,\ B_i\cap B_{i_1}\neq\varnothing\}}\alpha_i b_i
$$
which finishes the proof of Lemma~\ref{T12}.
\end{proof}

Now we can complete the proof of the theorem. Let $B_{i_1}$ be a ball with the largest diameter and let
$$
I=\{i:\, B_i\cap B_{i_1}\neq\varnothing\}
\quad
\text{and}
\quad
I^c=\{i:\, B_i\cap B_{i_1}=\varnothing\}.
$$
We have
$$
\bigcup_{i\in I} B_i\subset 3B_{i_1}
\quad
\text{and}
\quad
\sum_{i\in I}\alpha_i b_i\geq\frac{b_{i_1}}{2}.
$$
The inclusion is a consequence of the triangle inequality and the fact that
$\diam B_{i_1}\geq\diam B_i$ for $i\in I$, while the inequality follows from Lemma~\ref{T12}.

If $E\setminus 3B_{i_1}=\varnothing$, then \eqref{eq17} yields
$$
E\subset 3B_{i_1}
\quad
\text{and}
\quad
b_{i_1}\leq 2\sum_{i\in I}\alpha_i b_i\leq
2\sum_{i=1}^N\alpha_i b_i \leq 2\sum_{i=1}^N a_i b_i \, ,
$$
and the theorem follows. 

Therefore, we may assume that  $E\setminus 3B_{i_1}\neq\varnothing$. Since the balls $B_i$, $i\in I$ have empty intersection with $E\setminus 3B_{i_1}$,
$$
\sum_{i\in I^c}\alpha_i\chi_{B_i}\geq \chi_{E\setminus 3B_{i_1}}
$$
and hence $\{(\alpha_i,B_i)\}_{i\in I^c}$ is a weighted covering of $E\setminus 3B_{i_1}$ and the number of balls in that covering is less than or equal to $N-1$ (we removed at least one ball: $B_{i_1}$). According to the induction hypothesis, we can select pairwise disjoint balls 
$\{B_{i_j}\}_{j=2}^k$, $i_j\in I^c$ such that
$$
E\setminus 3B_{i_1}\subset\bigcup_{j=2}^k 3B_{i_j}
\quad
\text{and}
\quad
\sum_{j=2}^k b_{i_j}\leq
2\sum_{i\in I^c}\alpha_i b_i \, .
$$
Therefore,
$$
E\subset 3B_{i_1}\cup \bigcup_{j=2}^k 3B_{i_j}=\bigcup_{j=1}^k 3B_{i_j}
$$
(note that $B_{i_1}\cap B_{i_j}=\varnothing$, for $j\geq 2$ so the balls $\{ B_{i_j}\}_{j=1}^k$ are pairwise disjoint) and
$$
\sum_{j=1}^k b_{i_j} = b_{i_1}+\sum_{j=2}^k b_{i_j}\leq
2\sum_{i\in I}\alpha_i b_i +
2\sum_{i\in I^c}\alpha_i b_i
=
2\sum_{i=1}^N\alpha_i b_i \leq \sum_{i=1}^N a_i b_i \, .
$$
The proof is complete.
\end{proof}

\begin{corollary}
\label{T13}
Let $E$ be a non-empty subset of a metric space, $\{ b_i\}_{i=1}^\infty$, a sequence of non-negative numbers, and 
$\{(a_i,B_i)\}_{i=1}^\infty$, a weighted covering of $E$ by (all open or all closed) balls i.e., 
$$
\chi_E\leq\sum_{i=1}^\infty a_i\chi_{B_i},
\quad a_i\geq 0.
$$
Then there is a subfamily of balls $\{ B_{i_j}\}_{j=1}^\infty$ such that
$$
E\subset\bigcup_{j=1}^\infty 3B_{i_j}
\quad
\text{and}
\quad
\sum_{j=1}^\infty b_{i_j}\leq 8\sum_{i=1}^\infty a_i b_i.
$$
\end{corollary}
\begin{remark}
Differently than in Theorem~\ref{T11}, we do not assume that the balls $\{ B_{i_j}\}_{j=1}^\infty$ are pairwise disjoint.
\end{remark}
\begin{proof}
If $\sum_{i=1}^\infty a_ib_i=+\infty$, the claim is obvious. Therefore, we may assume that
$
M:=\sum_{i=1}^\infty a_ib_i<\infty.
$
We divide the series into finite blocks such that
$$
\sum_{i=1}^\infty a_i b_i = \sum_{k=0}^\infty \Big(\underbrace{\sum_{i=N_k+1}^{N_{k+1}} a_i b_i}_{\leq 4^{-k}M}\Big),
\quad
0=N_0<N_1<N_2<\ldots
$$
Let
$$
E_k=\left\{x\in E:\, \sum_{i=N_k+1}^{N_{k+1}} 2^{k+1}a_i\chi_{B_i}(x)\geq 1\right\}.
$$
Observe that $E=\bigcup_{k=0}^\infty E_k$. Indeed, if $x\in E$, then
$$
\sum_{k=0}^\infty\Big(\sum_{i=N_k+1}^{N_{k+1}} a_i\chi_{B_i}(x)\Big)=\sum_{i=1}^\infty a_i\chi_{B_i}(x)\geq \chi_E(x)=1=\sum_{k=0}^\infty 2^{-(k+1)}. 
$$
Therefore, there is $k$ such that
$$
\sum_{i=N_k+1}^{N_{k+1}} a_i\chi_{B_i}(x)\geq 2^{-(k+1)},
\quad
\text{so}
\quad x\in E_k.
$$
By the definition of $E_k$, the family $\{(2^{k+1}a_i,B_i)\}_{i=N_k+1}^{N_{k+1}}$ is a finite weighted covering of $E_k$. According to Theorem~\ref{T11}, we can select pairwise disjoint balls
$\{B_{i_j}^{(k)}\}_{j=1}^{\ell_k}$ from $\{ B_i\}_{i=N_k+1}^{N_{k+1}}$ so that
$$
E_k\subset\bigcup_{j=1}^{\ell_k} 3B_{i_j}^{(k)}
\quad
\text{and}
\quad
\sum_{j=1}^{\ell_k} b_{i_j}^{(k)}\leq
2\sum_{i=N_k+1}^{N_{k+1}} 2^{k+1}a_i b_i<4\cdot 2^{-k}M.
$$
To be more precise, we select this family of balls only if $E_k\neq\varnothing$. If $E_k=\varnothing$, we select empty family of balls.

If we relabel balls as
$$
\{B_{i_j}^{(k)}:\, k\in\bbbn\cup\{ 0\},\ 1\leq j\leq\ell_k\}:=\{ B_{i_j}\}_{j=1}^\infty,
$$
then
$$
E=\bigcup_{k=0}^\infty E_k\subset\bigcup_{k=0}^\infty\bigcup_{j=1}^{\ell_k} 3B_{i_j}^{(k)} =
\bigcup_{j=1}^\infty 3B_{i_j},
$$
and
$$
\sum_{j=1}^\infty b_{i_j} =\sum_{k=0}^\infty\sum_{j=1}^{\ell_k} b_{i_j}^{(k)}\leq\sum_{k=0}^\infty 4\cdot 2^{-k}M=8M=8\sum_{i=1}^\infty a_i b_i.
$$
\end{proof}

\section{Proof of Theorem~\ref{T7}}
\label{3.11}
First we will prove \eqref{eq21}. Note that the inequality $\lambda_\delta^s(E)\leq \H_\delta^s(E)$ is obvious and follows upon taking weighted coverings with coefficients $a_i=1$ so it remains to prove that
$\H^s_{6\delta}(E)\leq 8\cdot 6^s\lambda_\delta^s(E)$.

Let $\{(a_i,A_i)\}_{i=1}^\infty$ be a weighted $\delta$-covering of $E$,
$$
\chi_E\leq\sum_{i=1}^\infty a_i\chi_{A_i},
\quad
a_i\geq 0,
\quad
\diam A_i\leq\delta.
$$
Each of the sets $A_i$ is contained in a closed ball $B_i$ of radius $\diam A_i$. Hence
$$
\diam (3B_i)\leq 6\diam A_i\leq 6\delta
\quad
\text{so}
\quad
\zeta^s(3B_i)\leq 6^s\zeta^s(A_i).
$$
Since $\{ (a_i,B_i)\}_{i=1}^\infty$ is also a weighted cover of $E$, Corollary~\ref{T13}  with
$b_i=\zeta^s(A_i)$ yields a subfamily $\{ B_{i_j}\}_{j=1}^\infty$ of balls such that
$$
E\subset \bigcup_{j=1}^\infty 3B_{i_j}
\quad
\text{and}
\quad
\sum_{j=1}^\infty \zeta^s(A_{i_j})\leq 8\sum_{i=1}^\infty a_i\zeta^s(A_i).
$$
Therefore,
$$
\H^s_{6\delta}(E)\leq \sum_{j=1}^\infty
\zeta^s(3B_{i_j})\leq
6^s\sum_{j=1}^\infty \zeta^s(A_{i_j})\leq
8\cdot 6^s\sum_{i=1}^\infty a_i\zeta^s(A_i)
$$
and taking the infimum over all weighted $\delta$-coverings 
$\{(a_i,A_i)\}_{i=1}^\infty$ of $E$ proves that $\H^s_{6\delta}(E)\leq 8\cdot 6^s\lambda_\delta^s(E)$ and completes the proof of \eqref{eq21}.

Passing to the limit in \eqref{eq21} as $\delta\to 0^+$ yields
$$
(8\cdot 6^s)^{-1}\H^s(E)\leq\lambda^s(E)\leq \H^s(E).
$$
This proves \eqref{eq22} when $\H^s(E)=\infty$. Therefore, it remains to prove
\begin{equation}
\label{eq39}    
\H^s(E)\leq\lambda^s(E)
\quad
\text{assuming that $\H^s(E)<\infty$.}
\end{equation}
Let $\tilde{E}$ be a Borel set such that $E\subset\tilde{E}$ and $\H^s(\tilde{E})=\H^s(E)$. 

Fix $\eps>0$. For each $j\in\bbbn$, let $W_j$ be the set of points $x\in\tilde{E}$ such that
$$
x\in F\subset\bar{B}(x,1/j)
\quad
\Longrightarrow
\quad
\H^s(\tilde{E}\cap F)\leq (1+\eps)\zeta^s(F).
$$
Note that $W_1\subset W_2\subset\ldots$ and Lemma~\ref{T6} applied to $\tilde{E}$ regarded as a metric space yields (because $\H^s(\tilde{E})<\infty$)
$$
\H^s\Big(\tilde{E}\setminus\bigcup_{j=1}^\infty W_j\Big)=0.
$$
Therefore, Lemma~\ref{T34} implies
$$
\H^s(E)\leq \H^s\Big(E\cap\bigcup_{j=1}^\infty W_j\Big)+
\underbrace{\H^s\Big(E\setminus\bigcup_{j=1}^\infty W_j\Big)}_{0}=
\lim_{j\to\infty} \H^s(E\cap W_j).
$$
It remains to show that
\begin{equation}
\label{eq42}
\H^s(E\cap W_j)\leq (1+\eps)(\lambda_{1/j}^s(E)+\eps)
\end{equation}
as passing to the limit as $j\to\infty$ and then as $\eps\to 0^+$ will imply \eqref{eq39}.

Fix $j\in\bbbn$. Let $\{(a_k,A_k)\}_{k=1}^\infty$ be a weighted $1/j$-covering of $E$ by closed sets such that
\begin{equation}
\label{eq40}
\sum_{k=1}^\infty a_k\zeta^s(A_k)\leq\lambda_{1/j}^s(E)+\eps.
\end{equation}
Let $I=\{k:\, W_j\cap A_k\neq\varnothing\}$. We have
$$
\chi_{E\cap W_j}\leq\sum_{k\in I} a_k\chi_{\tilde{E}\cap A_k}.
$$
Let
$$
Z=\Big\{x:\, \sum_{k\in I} a_k\chi_{\tilde{E}\cap A_k}(x)\geq 1\Big\}.
$$
The set $Z$ is Borel, $E \cap W_j \subset Z$, and
$$
\chi_Z\leq \sum_{k\in I} a_k\chi_{\tilde{E}\cap A_k}\, .
$$
Integrating this inequality with respect to $\H^s$ yields
$$
\H^s(E\cap W_j)\leq\H^s(Z)\leq \sum_{k\in I} a_k\H^s(\tilde{E}\cap A_k).
$$
If $k\in I$, then there is $x\in W_j\cap A_k$ and hence
$$
x\in A_k\subset\bar{B}(x,1/j)
\quad
\text{so}
\quad
\H^s(\tilde{E}\cap A_k)\leq (1+\eps)\zeta^s(A_k)
$$
by the definition of the set $W_j$. Therefore,
$$
\H^s(E\cap W_j)\leq (1+\eps)
\sum_{k\in I} a_k\zeta^s(A_k)\leq
(1+\eps)(\lambda_{1/j}^s(E)+\eps),
$$
where the last inequality follows from \eqref{eq40}.
This proves \eqref{eq42} and completes the proof of the theorem.

\hfill $\Box$

\section{Proof of Theorem~\ref{T5}}
\label{3.13}
We first prove the following easier inequality 
\begin{equation}
\label{eq6}
\int^\bullet_X f \ d\mathcal{H}^s \leq \int^*_X f \ d\mathcal{H}^s \, .
\end{equation}
To this end it suffices to prove that for any $\delta>0$
\begin{equation}
\label{eq31}
\int^\bullet_X f \ d\mathcal{H}^s_\delta \leq \int^*_X f \ d\mathcal{H}^s \, ,
\end{equation}
as \eqref{eq6} will follow upon passing to the limit as $\delta\to 0^+$.
Assume that the right-hand side of \eqref{eq31} is finite. Given $\eps>0$, it follows from Lemma~\ref{T2} that there is a step function
$$
f\leq\sum_{i=1}^\infty a_i\chi_{A_i},
\quad 
a_i>0
$$
such that
$$
\sum_{i=1}^\infty a_i\H^s_\delta(A_i)\leq\sum_{i=1}^\infty a_i\H^s(A_i)\leq\int_X^*f\, d\H^s+\frac{\eps}{2}.
$$
For each $i$, there is a $\delta$-covering
$A_i\subset\bigcup_{j=1}^\infty A_{ij}$,
satisfying
$$
\sum_{j=1}^\infty\zeta^s(A_{ij})<\H^s_\delta(A_i)+ \frac{\eps}{2^{i+1}a_i}.
$$
Then with $a_{ij}=a_i$,
$$
f\leq\sum_{i=1}^\infty a_i \chi_{A_i}
\leq\sum_{i=1}^\infty a_i\Big(\sum_{j=1}^\infty \chi_{A_{ij}}\Big)=\sum_{i,j=1}^\infty a_{ij}\chi_{A_{ij}}
$$
so
$\{(a_{ij},A_{ij})\}_{i,j=1}^\infty$ is a weighted $\delta$-covering of $f$ and hence 
\begin{equation*}
\begin{split}
&\int_X^\bullet f\, d\H_\delta^s 
\leq 
\sum_{i,j=1}^\infty a_{ij}\zeta^s(A_{ij})=
\sum_{i=1}^\infty a_i\Big(\sum_{j=1}^\infty\zeta^s(A_{ij})\Big)\\
&<
\sum_{i=1}^\infty a_i\Big(\H^s_\delta(A_i)+\frac{\eps}{2^{i+1}a_i}\Big)
=
\sum_{i=1}^\infty a_i\H_\delta^s(A_i)+\frac{\eps}{2}\leq
\int_X^* f\, d\H^s+\eps.
\end{split}    
\end{equation*}
Since $\eps>0$ was chosen arbitrarily, \eqref{eq31} and hence \eqref{eq6} follow.

Now we must prove the reverse inequality
\begin{equation}
\label{eq9}
 \int_X^* f \, d\H^s \leq \int_X^\bullet f \, d\H^s\, .
\end{equation}
Clearly, it is important to consider the set
$A = \{x \in X: f(x) > 0 \}$, where the function $f$ is positive. We will split the proof into three cases. We shall also assume that the right-hand side in \eqref{eq9} is finite.

\noindent
{\sc Case 1.} $\H^s(A)<\infty$.

This case is similar to the proof of \eqref{eq22}.
Let $\eps>0$ be given. 
According to Theorem~\ref{T33},
there is a Borel set $\tilde{A}$ such that $A\subset\tilde{A}$ and $\H^s(A)=\H^s(\tilde{A})$. Applying Lemma~\ref{T6} to $\tilde{A}$ regarded as a metric space, we have that there is a set $E\subset\tilde{A}$, $\H^s(E)=0$, such that
$$
\forall\ x\in\tilde{A}\setminus E\ \ \
\exists\ \delta_x>0\ \ 
\forall\ F\subset X\ \ 
(x\in F\subset \bar{B}(x,\delta_x)\ \Rightarrow\ \H^s(\tilde{A}\cap F)\leq (1+\eps)\zeta^s(F)).
$$
Let $W_j\subset\tilde{A}$ be the set of points $x \in \tilde{A} $ such that
\begin{equation}
\label{eq11}    
x\in F\subset\bar{B}(x,1/j)
\quad
\Longrightarrow
\quad
\H^s(\tilde{A}\cap F)\leq (1+\eps)\zeta^s(F).
\end{equation}
Clearly, $W_1\subset W_2\subset\ldots$ and
$$
\tilde{A}=E\cup\bigcup_{j=1}^\infty W_j,
\quad
\H^s(E)=0.
$$
It suffices to prove that for each $j$, we have
\begin{equation}
\label{eq12}
\int_X^* f\chi_{W_j}\, d\H^s\leq
(1+\eps)\left(\int_X^\bullet f\, d\H^s_{1/j}+\eps\right),
\end{equation}
because, \eqref{eq9} will follow from Lemma~\ref{T1} upon passing to the limit, first as $j\to\infty$, and then as $\eps\to 0^+$.

According to the definition of the weighted integral and Remark~\ref{R1}, for each $j$ there is a weighted $1/j$-covering
$$
f(x)\leq\sum_{k=1}^\infty a_k\chi_{A_{jk}}(x),
\quad
\text{$A_{jk}$-closed,}
\quad
\diam A_{jk}\leq \frac{1}{j}
$$
such that
$$
\sum_{k=1}^\infty a_k\zeta^s(A_{jk})\leq \int_X^\bullet f\, d\H^s_{1/j}+\eps.
$$
Let $I=\{k:\, W_j\cap A_{jk}\neq\varnothing\}$. We have
$$
f\chi_{W_j}\leq\sum_{k\in I} a_k\chi_{\tilde{A}\cap A_{jk}}
$$
and measurability of the right hand side yields
$$
\int_X^* f\chi_{W_j}\, d\H^s\leq\sum_{k\in I} a_k\H^s(\tilde{A}\cap A_{jk})\leq\heartsuit.
$$
If $k\in I$, and $x\in W_j\cap A_{jk}$, then $x\in A_{jk}\subset\bar{B}(x,1/j)$ so \eqref{eq11} yields
$$
\heartsuit\leq
(1+\eps)\sum_{k\in I} a_k\zeta^s(A_{jk})\leq (1+\eps)\Big(\int_X^\bullet f\, d\H^s_{1/j}+\eps\Big).
$$
This completes the proof of \eqref{eq12}.

\noindent
{\sc Case 2.} {\em $A=\bigcup_{i=1}^\infty A_i$, where $\H^s(A_i)<\infty$.}

By replacing  $A_i $ with $\bigcup_{1 \leq j \leq i} {A}_j$, we can assume further that $A_1 \subset A_2 \subset \ldots$ Since $\H^s(\{x: (f\chi_{A_i})(x)>0\})<\infty$, inequality \eqref{eq9} follows from Case~1 applied to $f\chi_{A_i}$ and from Lemma~\ref{T1}:
$$
\int_X^*f\, d\H^s\stackrel{i\to\infty}{\xleftarrow{\hspace*{0.8cm}}} 
\int_X^*f\chi_{A_i}\, d\H^s\leq\int_X^\bullet f\chi_{A_i}d\, \H^s
\leq
\int_X^\bullet f\, d\H^s.
$$

\noindent
{\sc Case 3.} {\em The measure $\H^s$ of the set $A$ is not $\sigma$-finite.}

In order to prove inequality \eqref{eq9}, it suffices to show that
\begin{equation}
\label{eq10}
\int_X^\bullet f \, d\H^s= \infty \, .
\end{equation} 
To prove this, we will use Theorem~\ref{T7}.
Since the $\H^s$ measure of the set $\{ f>0\}$ is not $\sigma$-finite, there is $t>0$ such that $\H^s(\{f\geq t\})=\infty$. Therefore, for every $M>0$, there is $\delta>0$ such that
$$
\H_{6\delta}^s(\{x\in X:\, f(x)\geq t\})>M
$$
so Theorem~\ref{T7} yields ($C=8\cdot 6^s$):
$$
\int_X^\bullet f\, d\H^s \geq 
\int_X^\bullet t \chi_{\{f\geq t\}}\, d\H^s_\delta = t\lambda_\delta^s(\{ f\geq t\}) \geq C^{-1}t\H^s_{6\delta}(\{ f\geq t\}) \geq  C^{-1}tM \, , 
$$
and \eqref{eq10} follows.
The proof of Theorem~\ref{T5} (and hence that of Theorem~\ref{T15}) is complete.

\section{Generalized Coarea Inequality}
\label{NC}

The next result is a generalization of Theorem~\ref{T14} and it is motivated by the results in \cite{azzams,davids,HKK}. Recall that $\Phi^{s,t}$ was defined in Definition~\ref{D1}.

\begin{theorem}
\label{T20}
If $f:X\to Y$ is a uniformly continuous map between metric spaces, $0\leq t\leq s<\infty$ and $E\subset X$, then
$$
\int_Y^* \H^{s-t}(f^{-1}(y)\cap E)\, d\H^t(y)\leq
\Phi^{t,s-t}(f,E).
$$
\end{theorem}
\begin{proof}
It follows immediately from Lemma~\ref{T17}, Theorem~\ref{T5} and Lemma~\ref{T1}.
\end{proof}

\begin{proof}[Proof of Theorem~\ref{T14}]
Theorem~\ref{T20} and Lemma~\ref{T16} imply inequality \eqref{eq29}
and it remains to show measurability of the function \eqref{eq28} under the assumptions that $X$ is boundedly compact, $E$ is $\H^s$-measurable and $\H^s(E)<\infty$.

This fact is standard, but for the sake of completeness we will provide a short proof.
Since bounded and closed sets are compact, Lemma~\ref{T36} implies existence of a decomposition
$$
E=N\cup\bigcup_{i=1}^\infty K_i,
\quad
\H^s(N)=0,
\quad
\text{$K_1\subset K_2\subset\ldots$ compact sets.}
$$
It follows from \eqref{eq29} that
$\H^{s-t}(f^{-1}(y)\cap N)=0$ for $\H^t$-almost every $y\in Y$ so for almost all $y\in Y$ we have
$$
\H^{s-t}(f^{-1}(y)\cap E)=\H^{s-t}\Big(f^{-1}(y)\cap\bigcup_{i=1}^\infty K_i\Big)
= \lim_{i\to\infty} \H^{s-t}(f^{-1}(y)\cap K_i).
$$
Therefore it remains to show measurability of the function $y\mapsto\H^{s-t}(f^{-1}(y)\cap K)$, where $K\subset X$ is a compact set. To this end it suffices to prove measurability of the sets 
$$
Y_u=\{y\in Y:\, \H^{s-t}(f^{-1}(y)\cap K)\leq u\},
\quad 
u\in\bbbr.
$$
If $u<0$, $Y_u=\varnothing$ so we may assume that $u\geq 0$. 

Recall that in Section~\ref{HM} the content $\mathscr{H}_\delta^{s-t}$ was defined with  open sets. Since it defines the standard Hausdorff measure, we have
$$
Y_u=\bigcap_{j=1}^\infty\Big\{y\in Y:\, \mathscr{H}^{s-t}_{1/j}(f^{-1}(y)\cap K)<u+\frac{1}{j}\Big\}
$$
so it suffices to show that the sets of the form
$$
V=\{y\in Y:\, \mathscr{H}^{s-t}_{\delta}(f^{-1}(y)\cap K)<v\}
$$
are open (for $v$ and $\delta$ positive values). To this end it suffices to show that if $y\in V$ and $y_k\to y$, then $y_k\in V$ for sufficiently large $k$.
For $y\in V$ fix an open covering
$$
f^{-1}(y)\cap K\subset\bigcup_{j=1}^\infty U_{j},
\quad
\diam U_j<\delta,
\quad
\sum_{j=1}^\infty\zeta^{s-t}(U_j)<v.
$$
Using a standard compactness argument, it follows that there exists a $k_0$ such that $f^{-1}(y_k)\cap K\subset\bigcup_{j=1}^\infty U_{j}$ for $k\geq k_o$ and hence 
$\mathscr{H}_\delta^{s-t}(f^{-1}(y_k)\cap K)<v$, proving that $y_k\in V$ for $k\geq k_o$.
\end{proof}

\subsection{The lower density and doubling spaces}
\label{123}
Throughout Section~\ref{123}, $X$ and $Y$ will denote metric spaces.
In this section we will improve Theorem~\ref{T20} under the assumption that the Hausdorff measure on $X$ is doubling. The main result of this section, Theorem~\ref{T30}, is closely related to the coarea formula, see Corollary~\ref{T28} and Remark~\ref{R7}.
\begin{definition}
For an arbitrary map $f:E\to Y$, $E\subset X$, 
$s\in (0,\infty)$, 
$t\in [0,\infty)$ and $\delta\in (0,\infty]$ we define
$$
\tilde{\H}^{s,t}_\delta(f,E)=\inf\sum_{i=1}^\infty \H^s_\infty(f(A_i)) \zeta^t(A_i),
$$
where the infimum is taken over all $\delta$-coverings $\{ A_i\}_{i=1}^\infty$ of $E$. If no such covering exists then $\tilde{\H}^{s,t}_\delta(f,E) = \infty$.
\end{definition}
The following elementary observation will be useful.
\begin{lemma}
\label{T27}
For any map $f:E\to Y$, $E\subset X$, $s\in (0,\infty)$, $t\in [0,\infty)$ and $\delta\in (0,\infty]$ we have
$$
\Phi^{s,t}_\delta(f,E)=\tilde{\H}^{s,t}_\delta(f,E).
$$
\end{lemma}
\begin{proof}
Since $\H^s_\infty(f(A_i))\leq\zeta^s(f(A_i))$, the inequality $\tilde{\H}^{s,t}_\delta\leq\Phi^{s,t}_\delta$ is obvious. Therefore, it remains to prove that $\Phi^{s,t}_\delta(f,E)\leq\tilde{\H}^{s,t}_\delta(f,E)$ and we can assume that $\tilde{\H}^{s,t}_\delta(f,E)<\infty$.

Given $\eps>0$, let $\{ A_i\}_{i=1}^\infty$ be a $\delta$-covering of $E$ such that
$$
\sum_{i=1}^\infty \zeta^t(A_i)\H^s_\infty(f(A_i))<\tilde{\H}^{s,t}_\delta(f,E)+\frac{\eps}{2}.
$$
For each $i\in\bbbn$, let $\{C_{ij}\}_{j=1}^\infty$ be a covering of $f(A_i)$ such that
$$
\sum_{j=1}^\infty \zeta^s(C_{ij})<\H^s_\infty(f(A_i))+\frac{\eps}{2^{i+1}(\zeta^t(A_i)+1)}.
$$
Let $A_{ij}=A_i\cap f^{-1}(C_{ij})$. Then
\begin{equation*}
\begin{split}
\Phi^{s,t}_\delta(f,E)
&\leq
\sum_{i,j=1}^\infty \zeta^t(A_{ij})\zeta^s(f(A_{ij}))
\leq
\sum_{i=1}^\infty \zeta^t(A_i)\Big(\sum_{j=1}^\infty\zeta^s(C_{ij})\Big)\\
&\leq
\sum_{i=1}^\infty\zeta^t(A_i)\H^s_\infty(f(A_i))+\frac{\eps}{2}
<
\tilde{\H}^{s,t}_\delta(f,E)+\eps
\end{split}    
\end{equation*}
and the result follows.
\end{proof}

\begin{definition}
Let $X$ and $Y$ be metric spaces, $E\subset X$ any subset, and  $s>0$.
For any mapping $f:E\to Y$, we define the {\em lower $s$-density of $f$} as
$$
\Theta^{s}_*(f,E,x)=\liminf_{r\to 0^+} \frac{\H^s_\infty(f(B(x,r)\cap E))}{\omega_sr^s}\, .
$$
\end{definition}
\begin{remark}
\label{R8}
It is a routine exercise to show that we can replace open balls by closed balls in the definition of the lower density i.e.,
$$
\Theta^{s}_*(f,E,x)=\liminf_{r\to 0^+} \frac{\H^s_\infty(f(\bar{B}(x,r)\cap E))}{\omega_sr^s}\, .
$$
\end{remark}
\begin{remark}
\label{R9}
Note that if $f$ is Lipschitz, then
$\Theta^s_*(f,E,x)\leq (\lip f)^s$.
\end{remark}

\begin{remark}
In the case when $X=\bbbr^{n+m}$, $s=n$, and $Y$ is any metric space, the lower (and upper) $n$-density of $f$ was introduced in \cite{HZ} and it played an important role in the implicit function theorem for Lipschitz mappings into metric spaces.
\end{remark}
\begin{definition}
We say that a Borel measure $\mu$ on $X$ is doubling if 
$0<\mu(B(x,r))<\infty$ for all $x\in X$ and $r>0$, and if there is a constant $C>0$ such that
$\mu(B(x,2r))\leq C\mu(B(x,r))$ for all $x\in X$ and $r>0$.
\end{definition}
The next definition provides a particularly important instance of a doubling measure.
\begin{definition}
\label{D2}
We say that the Hausdorff measure $\H^s$, $s>0$, on $X$ is {\em Ahlfors regular}, if 
there are constants $C_A, C_B>0$ such that $C_Ar^s\leq\H^s(B(x,r))\leq C_Br^s$ for all $x\in X$ and all $r<\diam X$.
\end{definition}
\begin{definition}
We say that a metric space is {\em metric doubling} if there is $M>0$ such that every ball $B$ can be covered by no more than $M$ balls of half the radius.
\end{definition}
Note that if a metric space is metric doubling, then bounded sets are totally bounded. Recall that a metric space is compact if and only it it is complete and totally bounded. Therefore we have
\begin{lemma}
\label{T37}
If $X$ is metric doubling and complete, then $X$ is boundedly compact.
\end{lemma}
The following lemma is an easy exercise
\begin{lemma}
If $\mu$ is a doubling measure on $X$, then $X$ is metric doubling.
\end{lemma}
Indeed, there cannot be too many points in $B$ whose mutual distances are greater than or equal to $r/2$, where $r$ is the radius of $B$.

The next result is the Vitali covering theorem for doubling measures, see \cite[Theorem~1.6]{heinonen}
\begin{lemma}
\label{T29}
Let $\mu$ be a doubling measure on a metric space $X$ and let $E\subset X$.   
If $\mathcal{F}$ is a family of closed balls centered at $E$ such that for every $x\in E$
$$
\inf\{r>0:\, B(x,r)\in\mathcal{F}\}=0,
$$
then there is a countable subfamily $\{B_1,B_2,\ldots\}\subset\mathcal{F}$
of pairwise disjoint balls
such that 
$$
\mu\Big(E\setminus \bigcup_{i=1}^\infty B_i\Big)=0.
$$
\end{lemma}
The next result is  the Lebesgue differentiation theorem for doubling measures. It is a consequence of Lemma~\ref{T29}, see \cite[Theorem~1.8]{heinonen}
\begin{lemma}
\label{T31}
If $g$ is a locally integrable function on a metric space with a doubling measure $\mu$, then
\begin{equation}
\label{eq35}
\lim_{r\to 0}\mvint_{B(x,r)} g\, d\mu=g(x)
\quad
\text{for $\mu$-almost all $x\in X$.}
\end{equation}
\end{lemma}

\begin{lemma}
\label{T32}
Suppose the metric space $X$ is metric doubling and $E\subset X$ is bounded. If $s,t\in [0,\infty)$, and
$f:E\to Y$ is a mapping, then
$\Phi^{s,t}(f,E)=0$ if and only if $\Phi^{s,t}_\infty(f,E)=0$.
\end{lemma}
\begin{proof}
Since $\Phi^{s,t}_\infty\leq\Phi^{s,t}$, one implication is obvious and it remains to show that if 
$\Phi^{s,t}_\infty(f,E)=0$, then for any $\delta>0$ we have $\Phi^{s,t}_\delta(f,E)=0$.
Since $E$ is bounded and $X$ is metric doubling, $E$ can be split into a finite number of pieces, say $N(\delta)$ many, each of diameter less than $\delta$. 

Given $\eps>0$, let $E\subset\bigcup_{i=1}^\infty A_i$ be a covering such that
$$
\sum_{i=1}^\infty\zeta^s(f(A_i))\zeta^t(A_i)<\frac{\eps}{N(\delta)}.
$$
By replacing $A_i$ with $E\cap A_i$ we can further assume that $A_i\subset E$. Each of the sets $A_i$ is a union of  $N(\delta)$ sets $\{A_{ij}\}_{j=1}^{N(\delta)}$, each of diameter less that $\delta$. Therefore,
$$
\Phi^{s,t}_\delta(f,E)\leq\sum_{i=1}^\infty\sum_{j=1}^{N(\delta)}
\zeta^{s}(f(A_{ij}))\zeta^t(A_{ij})\leq N(\delta)\sum_{i=1}^\infty\zeta^s(f(A_i))\zeta^t(A_i)<\eps.
$$
\end{proof}
\begin{theorem}
\label{T30}
Suppose $0<t\leq s<\infty$, the measure $\H^s$ is Ahlfors regular on a complete metric space $X$, $E\subset X$ is closed, and $f:E\to Y$ is Lipschitz. Then 
\begin{equation}
\label{eq33}
\int_Y^* \H^{s-t}(f^{-1}(y) \cap E) \, d\H^t(y) \leq \frac{\omega_{s-t}\omega_t}{C_A}\int_E \Theta^{t}_*(f,E,x) \, d\H^s(x)  \, .
\end{equation}
where $C_A$ is the constant from Definition~\ref{D2}.
\end{theorem}
\begin{remark}
The assumption that $X$ is complete guarantees that $X$ is boundedly compact (Lemma~\ref{T37}).
Since $E$ is closed, $\bar{B}(x,r)\cap E$ is compact. We need this assumption to prove measurability of $\Theta_*^t(f,E,\cdot)$.
We do not know if the theorem is true for any $\H^s$-measurable set $E$, and without assuming that $X$ is complete.
\end{remark}
\begin{proof}
We can assume that $E$ is bounded, because the general case will follow from the inequality applied to $E\cap \bar{B}(x_o,R)$ upon passing to the limit as $R\to\infty$. Note that in order to pass to the limit on the left hand side, we need to use Lemma~\ref{T1}.

The density function $\Theta^{t}_*(f,E,\cdot)$ is measurable.
To see this it suffices to prove that the function $h_r(x)=\H^s_\infty(f(\bar{B}(x,r)\cap E))$ (see Remark~\ref{R8}) is Borel
and this is true since the function is upper-semicontinuous meaning that $\limsup_{y\to x} h_r(y)\leq h_r(x)$. Indeed, under our assumptions, the set $\bar{B}(x,r)\cap E$ and its image are compact. We can approximate $\H^s_\infty(f(\bar{B}(x,r)\cap E))$ using an open covering $\{ U_i\}_{i=1}^\infty$. If $y$ is close to $x$, then $f(\bar{B}(y,r)\cap E)\subset\bigcup_{i=1}^\infty U_i$ and we can use the same open covering $\{ U_i\}_{i=1}^\infty$ to get the upper estimate for the content $\H^s_\infty(f(\bar{B}(y,r)\cap E))$.

Since $\H^s(E)<\infty$ ($E$ is bounded and $\H^s$ is Ahlfors regular), in view of Remark~\ref{R9}, the right hand side of \eqref{eq33} is finite.

According to Theorem~\ref{T20}, it suffices to prove that
\begin{equation}
\label{eq34}
\Phi^{t,s-t}(f,E)\leq \frac{\omega_{s-t}\omega_t}{C_A} \int_E \Theta^{t}_*(f,E,x) \, d\H^s(x)  \, .
\end{equation}
Let $N$ be the set of points $x\in E$ for which \eqref{eq35} does not hold with
$g=\Theta^{t}_*(f,E,\cdot)\chi_E$. Since $\H^s(N)=0$, Lemma~\ref{T16} yields that
$\Phi^{t,s-t}(f,N)=0$ and hence by Lemma \ref{T35},
\begin{equation}
\label{eq41}
\Phi^{t,s-t}(f,E)=\Phi^{t,s-t}(f,N) \, .
\end{equation}

Given $\eps>0$ and $\delta>0$, for each $x\in E\setminus N$, there is a sequence $r_{x,i}\to 0^+$, $B_{x,i}=\bar{B}(x,r_{x,i})$ 
such that
$$
\frac{\H^t_\infty(f(E\cap B_{x,i}))}{\omega_t r_{x,i}^t}
\leq\Theta^{t}_*(f,E,x)+\frac{\eps}{2}\leq
\mvint_{B(x,r_{x,i})}\Theta^{t}_*(f,E,z)\chi_E(z)\, d\H^s(z)+\eps.
$$
Lemma~\ref{T29} applied to the family
$\{B_{x,i}: x \in E\setminus N, r_{x,i} < \delta/2 \}$ gives pairwise disjoint balls $B_i$ with diameters less than $\delta$ such that
$$
\H^s\Big(E\setminus\bigcup_{i=1}^\infty B_i\Big)=\H^s\Big((E\setminus N)\setminus\bigcup_{i=1}^\infty B_i\Big)=0.
$$
Using a similar argument as in the proof of \eqref{eq41}, one can easily show that
$$
\Phi^{t,s-t}_\delta(f,E)=\Phi^{t,s-t}_\delta\Big(f,E\cap\bigcup_{i=1}^\infty B_i\Big).
$$
Therefore, Lemma~\ref{T27} yields
\begin{equation*}
\begin{split}
&\Phi^{t,s-t}_\delta(f,E)=\Phi^{t,s-t}_\delta\Big(f,E\cap \bigcup_{i=1}^\infty B_{i}\Big)=
\tilde{\H}^{t,s-t}_\delta\Big(f,E\cap \bigcup_{i=1}^\infty B_{i}\Big)\\
&\leq
\sum_{i=1}^\infty\zeta^{s-t}(E\cap B_{i})\H^t_\infty(f(E\cap B_{i}))\\
&\leq
\sum_{i=1}^\infty \frac{\omega_{s-t}}{2^{s-t}}(2r_{i})^{s-t}\omega_t r_{i}^t
\Big(\mvint_{B(x_i,r_{i})}\Theta^{t}_*(f,E,z)\chi_E(z)\, d\H^s(z)+\eps\Big)\\
&\leq
\frac{\omega_{s-t}\omega_t}{C_A}
\Big(\int_E\Theta^{t}_*(f,E,z)\, d\H^s(z)+\eps\Big)
\end{split}    
\end{equation*}
and the result follows by letting $\delta\to 0^+$ and then $\eps\to 0^+$.
\end{proof}

It was proved in \cite[Proposition~5.2]{HZ} that if $f:E\to\bbbr^m$ is a Lipschitz continuous map defined on a measurable set $E\subset\bbbr^{n}$, $n\geq m$, then
$\Theta_*^m(f,E,x)=|J^mf|(x)$, where 
$$
|J^mf|(x)=\sqrt{\det(Df)(Df)^T}
\quad
\text{is the Jacobian.}
$$
This and the above result gives
\begin{corollary}
\label{T28}
If $f:E\to\bbbr^m$ is a Lipschitz map defined on a measurable set $E\subset \bbbr^{n}$, $n\geq m$, then
$$
\int_{\bbbr^m} \H^{n-m}(f^{-1}(y)\cap E)\, d\H^m(y)\leq
\frac{\omega_{n-m}\omega_m}{\omega_n}\int_{E} |J^mf|(x)\, d\H^n(x).
$$
\end{corollary}
\begin{remark}
\label{R7}
The celebrated coarea formula \cite[Theorem~3.10]{evans}, \cite[Theorem~3.2.11]{federer}, states that under the above assumptions
$$
\int_{\bbbr^m} \H^{n-m}(f^{-1}(y)\cap E)\, d\H^m(y)=\int_E |J^mf|(x)\, d\H^n(x).
$$
\end{remark}
Since we obtained the Corollary~\ref{T28} as a consequence of rather general results valid in metric spaces, it is not surprising that the result is not as sharp as the coarea formula. On the other hand a localized version of Theorem~\ref{T14} would  suggest a much weaker inequality with $|Df|^m$ instead of $|J^mf|$ since $|Df|$ can be regarded as a local Lipschitz constant of $f$. This shows that Theorem~\ref{T30} and hence also Theorem~\ref{T20} are substantial improvements of the coarea inequality.

\subsection{Mapping Content}
\label{MC}
In the context of quantitative decomposition of Lipschitz mappings into metric spaces Azzam and Schul \cite{azzams} defined the $(n,m)$-mapping content. This notion was further investigated by David and Schul \cite{davids} (see also \cite{HZ}).
\begin{definition}\label{bD1}
Let $Q_0=[0,1]^{n+m}$ be the unit cube and $X$ an arbitrary metric space. For 
a Lipschitz map $f:Q_0 \to X$ the {\em $(n,m)$-mapping content} of a set $E\subset Q_0$ is
$$
\H_\infty^{n,m}(f,E) =
\inf\sum_i \H^n_\infty(f(Q_i))  \zeta^m(Q_i)  ,
$$
where the infimum is over all coverings of $E$ by closed dyadic cubes with pairwise disjoint interiors.
\end{definition}
\begin{remark}
In fact their definition differs from ours
by a constant factor depending on $n$ and $m$ only.
\end{remark}
It follows directly from the definitions and from Lemma~\ref{T27} that
\begin{equation}
\label{eq36}
\Phi^{n,m}_\infty(f,E)=\tilde{\H}^{n,m}_\infty(f,E)\leq\H^{n,m}_\infty(f,E)
\end{equation}
and David and Schul \cite[Question~1.15]{davids} stated an open problem: 
$$
\text{Is it true that $\H^{n,m}_\infty(f,E)\leq C(n,m)\tilde{\H}^{n,m}(f,E)$?}
$$
As an application of Theorem~\ref{T20} we obtain:
\begin{corollary}
Suppose  $f:Q_0=[0,1]^{n+m}\to X$  is a Lipschitz mapping into a metric space and $E\subset Q_0$. If $ \H_\infty^{n,m}(f,E)=0$ then
$\H^m(f^{-1}(x) \cap E) = 0$ for $\H^n$-a.e. $x \in X$. 
\end{corollary}
\begin{proof}
It follows from \eqref{eq36} and from Lemma~\ref{T32} that $\Phi^{n,m}(f,E)=0$ and hence Theorem~\ref{T20} yields that
$$
\int_X^* \H^m(f^{-1}(x)\cap E)\, d\H^n(x)=0
$$
and the result follows from \eqref{eq37}.
\end{proof}


\begin{thebibliography}{888}

\bibitem{ambrosiot}
{\sc Ambrosio, L. Tilli, P.:} {\em Topics on analysis in metric spaces.} Oxford Lecture Series in Mathematics and its Applications, 25. Oxford University Press, Oxford, 2004.
%
\bibitem{azzams}
{\sc  Azzam, J., Schul, R.:}
Hard Sard: quantitative implicit function and extension theorems for Lipschitz maps. 
{\em Geom.\ Funct.\ Anal.} 22 (2012), 1062--1123.
%
\bibitem{buragoz}
{\sc Burago, Yu. D., Zalgaller, V. A.:} {\em Geometric inequalities.} Translated from the Russian by A. B. Sosinskiĭ. Grundlehren der Mathematischen Wissenschaften [Fundamental Principles of Mathematical Sciences], 285. Springer Series in Soviet Mathematics. Springer-Verlag, Berlin, 1988.
%
\bibitem{davies}
{\sc Davies, R. O.:}  Increasing sequences of sets and Hausdorff measure, 
{\em Proc.\ London Math.\ Soc.} 20 (1970), 222-236.
%
\bibitem{davids}
{\sc David, G. C., Schul, R.:}
Quantitative decompositions of Lipschitz mappings into metric spaces.
(Preprint 2020) arXiv:2002.10318v3. 
%
\bibitem{del}
{\sc Dellacherie, C.:}
{\em Ensembles Analytiques, Capacit\'es, Mesures de Hausdorff.} Lecture Notes in Math 295, Springer-Verlag, 1972.
%
\bibitem{eilenberg}
{\sc Eilenberg, S.:}  On $\Phi$ measures, {\em Annales de la Soci\'et\'e Polonaise de Mathemmatique},  17 (1938),  252--253.
%
\bibitem{eh}
{\sc  Eilenberg, S., Harrold, O. G., Jr.:} Continua of finite linear measure. I. 
{\em Amer.\ J. Math.} 65 (1943), 137--146.
\bibitem{evans}
%
{\sc Evans, L. C., Gariepy, R. F.:}
\emph{Measure theory and fine properties of functions.} 
Revised edition. Textbooks in Mathematics. CRC Press, Boca Raton, FL, 2015.
%
\bibitem{federer}
{\sc Federer, H.:} {\em Geometric measure theory.} Die Grundlehren der mathematischen Wissenschaften,
Band 153 Springer-Verlag New York Inc., New York 1969.
%
\bibitem{federer2}
{\sc Federer, H.:}
Some integralgeometric theorems.
{\em Trans.\ Amer.\ Math.\ Soc.} 77 (1954), 238--261.
%
\bibitem{HKK}
{\sc Haj\l{}asz, P., Korobkov, M. V., Kristensen, J.:}
A bridge between Dubovitskiĭ-Federer theorems and the coarea formula.
{\em J. Funct.\ Anal.} 272 (2017), 1265--1295.
%
\bibitem{HZ}
{\sc Haj\l{}asz, P., Zimmerman, S.:} An implicit function theorem for Lipschitz mappings into metric spaces
{\em  Indiana Univ.\ Math.\ J.} 69 (2020), 201-224.
%
\bibitem{heinonen}
{\sc Heinonen, J.:}
{\em Lectures on analysis on metric spaces.} Universitext. Springer-Verlag, New York, 2001.
%
\bibitem{howroyd}
{\sc  Howroyd, J. D.:} 
On dimension and on the existence of sets of finite positive Hausdorff measure. 
{\em Proc.\ London Math.\ Soc.} 70 (1995), 581-604.
%
\bibitem{HW}
{\sc Hurewicz, W., Wallman, H.:} {\em Dimension Theory.} Princeton Mathematical Series, v. 4. Princeton University Press, Princeton, N. J., 1941.
%
\bibitem{kelly1}
{\sc Kelly, J. D.:} The increasing sets lemma, and the approximation of analytic sets from within by compact sets, for the measures generated by method III. 
{\em J. London Math.\ Soc.} 8 (1974), 29-43.
%
\bibitem{kelly2}
{\sc Kelly, J. D.:} A method for constructing measures appropriate for the study of Cartesian products. {\em Proc.\ London Math.\ Soc.}  26 (1973), 521-546.
%
\bibitem{KP}
{\sc Krantz, S. G., Parks, H. R.:}
{\em Geometric integration theory.}
Cornerstones. Birkh\"auser Boston, Inc., Boston, MA, 2008
%
\bibitem{maly}
{\sc Mal\'y, J.:} Coarea integration in metric spaces. 
NAFSA 7-{\em Nonlinear analysis, function spaces and applications.} Vol. 7, 148–192,  Czech.\ Acad.\ Sci., Prague, 2003. 
%
\bibitem{nazarov}
{\sc Mathoverflow.}
Bounding an “integral” from below by the Hausdorff measure of the domain.
https://mathoverflow.net/q/355214/121665.
%
\bibitem{MO2}
{\sc Mathoverflow}
Unknown work of N\"obeling on topological/Hausdorff dimension.
\text{https://mathoverflow.net/q/360384/121665}
%
\bibitem{mattila2}
{\sc Mattila, P.:} Personal communication.
%
\bibitem{mattila}
{\sc Mattila, P.:}
{\em Geometry of sets and measures in Euclidean spaces. Fractals and rectifiability.} Cambridge Studies in Advanced Mathematics, 44. Cambridge University Press, Cambridge, 1995.
%
\bibitem{menger}
{\sc Menger, K.:} {\em Ergebnisse eines mathematischen Kolloquiums.} (German) [Results of a mathematical colloquium] With contributions by J. W. Dawson, Jr., R. Engelking and W. Hildenbrand, a foreword by G. Debreu and an afterword in English by F. Alt. Edited by E. Dierker and K. Sigmund. Springer-Verlag, Vienna, 1998.
%
\bibitem{nobeling}
{\sc N\"obeling, G.}
Hausdorffsche und mengentheoretische Dimension.
{\em Ergebnisse math.\ Kolloquium Wien} 3 (1931) 24-25 .
%
\bibitem{reichel}
{\sc Reichel, L. P.:} 
{\em The coarea formula for metric space valued maps.} Ph.D. thesis, ETH Z\"urich, 2009.
%
\bibitem{simon}
{\sc Simon, L.:} {\em Lectures on geometric measure theory.} Proceedings of the Centre for Mathematical Analysis, Australian National University, 3. Australian National University, Centre for Mathematical Analysis, Canberra, 1983.
%
\bibitem{szpilrajn}
{\sc Szpilrajn E.:} La dimension et la mesure, {\em Fund.\ Math.} 28 (1937), 81--89.
%


\end{thebibliography}
\end{document}